\documentclass[12pt]{amsart}

\usepackage[english]{babel}
\usepackage{mathrsfs,amssymb}
\usepackage{mathtools}
\usepackage[colorlinks, citecolor = blue]{hyperref}

\usepackage[shortlabels]{enumitem}
\setlist[itemize]{leftmargin=25pt}
\setlist[enumerate]{leftmargin=25pt}

\newtheorem{theorem}{Theorem}[section]
\newtheorem{lemma}[theorem]{Lemma}
\newtheorem{prop}[theorem]{Proposition}
\newtheorem{cor}[theorem]{Corollary}

\theoremstyle{definition}
\newtheorem{definition}[theorem]{Definition}
\newtheorem{que}[theorem]{Question}

\newtheorem{pr}[theorem]{Problem}

\theoremstyle{remark}
\newtheorem{remark}[theorem]{Remark}

\newtheorem{example}[theorem]{Example}

\numberwithin{equation}{section}
\usepackage[colorinlistoftodos,prependcaption,textsize=small]{todonotes}

\DeclareMathOperator*{\esssup}{ess\,sup}
\DeclareMathOperator*{\essinf}{ess\,inf}

\let \la=\lambda
\let \e=\varepsilon
\let \d=\delta

\let \a=\alpha
\let \f=\varphi
\let \b=\beta

\let \ga=\gamma
\let \D=\Delta

\allowdisplaybreaks

\begin{document}
\title[A boundedness criterion for the maximal operator]
{A boundedness criterion for the maximal operator on variable Lebesgue spaces}

\author[A.K. Lerner]{Andrei K. Lerner}
\address[A.K. Lerner]{Department of Mathematics,
Bar-Ilan University, 5290002 Ramat Gan, Israel}
\email{lernera@math.biu.ac.il}

\thanks{The author was supported by ISF grant no. 1035/21.}

\begin{abstract}
We obtain a necessary and sufficient condition on an exponent $p(\cdot)$ for which the Hardy--Littlewood maximal
operator is bounded on the variable $L^{p(\cdot)}$ space. It is formulated in terms of the Muckenhoupt-type condition $A_{p(\cdot)}$,
responsible for a local control of $p(\cdot)$, and a certain integral condition on $p(\cdot)$, responsible for the behaviour of $p(\cdot)$ at infinity.
Our approach is based on an earlier characterization established by L. Diening and on non-increasing rearrangements.
\end{abstract}



\maketitle

\section{Introduction}
Let $p(\cdot):{\mathbb R}^n\to [1,\infty)$ be a measurable function. Denote
by $L^{p(\cdot)}$ the space of real-valued measurable functions $f$ on ${\mathbb R}^n$ such
that
$$\|f\|_{L^{p(\cdot)}}:=\inf\left\{\la>0:\int_{{\mathbb R}^n}\left(\frac{|f(x)|}{\la}\right)^{p(x)}dx\le 1\right\}<\infty.$$

Define the Hardy--Littlewood maximal operator $M$ by
$$Mf(x):=\sup_{Q\ni x}\frac{1}{|Q|}\int_Q|f(y)|dy,$$
where the supremum is taken over all cubes $Q\subset {\mathbb R}^n$ containing the point $x$.

Let ${\mathcal P}$ denote the class of all exponents $p(\cdot)$ for which $M$ is bounded on $L^{p(\cdot)}$.
This paper is concerned with the following problem.

\begin{pr}\label{mq}
Find a constructive characterization of the class ${\mathcal P}$.
\end{pr}

In other words, we seek for a statement of the form ``$p(\cdot)\in {\mathcal P}$ if and only if condition $A$ holds", where $A$ is a checkable condition formulated in terms of $p(\cdot)$.

Observe that in 2005, Diening \cite{D1} obtained a remarkable characterization of the class ${\mathcal P}$ but it was not constructive. This result will play an important role
in our paper, and we will return to it a bit later. Meanwhile we provide a brief account of the main results related to Problem \ref{mq}. A more detailed exposition can be found in the monographs \cite{CUF, DHHR11}.

\subsection{Background} First we formulate sufficient conditions on $p(\cdot)$ in terms of the so-called log-H\"older continuity.

\begin{definition}\label{logh}
We say that a function $r(\cdot):{\mathbb R}^n\to {\mathbb R}$ is locally log-H\"older continuous,
and write $r(\cdot)\in LH_0$, if there exists a constant $C>0$ such that
$$|r(x)-r(y)|\le \frac{C}{-\log|x-y|}\quad(|x-y|<1/2,\,x,y\in {\mathbb R}^n).$$
We say that $r(\cdot)$ is log-H\"older continuous at infinity, and write $r(\cdot)\in LH_{\infty}$, if there exist constants $r_{\infty}$ and $C>0$ such that
$$|r(x)-r_{\infty}|\le \frac{C}{\log({\rm{e}}+|x|)}\quad(x\in {\mathbb R}^n).$$
\end{definition}

Given an exponent $p(\cdot)$, denote
$$p_-:=\displaystyle \operatornamewithlimits{ess\,
inf}_{x\in\mathbb{R}^n} p(x)\quad\text{and}\quad p_+:=\displaystyle
\operatornamewithlimits{ess\, sup}_{x\in\mathbb{R}^n} p(x).$$

The following theorem goes back to Diening \cite{D,D2} and Cruz-Uribe--Fiorenza--Neugebauer \cite{CUFN}.
It can be found in the monographs  \cite[Th. 3.16]{CUF} and \cite[Th. 4.3.8]{DHHR11}, where a detailed history is provided.

\begin{theorem}\label{sufc}
If $p_->1$ and $\frac{1}{p(\cdot)}\in LH_0\cap LH_{\infty}$, then $p(\cdot)\in {\mathcal P}$.
\end{theorem}

\begin{remark}\label{rem}
If $p_->0$ and $p_+<\infty$, then the condition $\frac{1}{p(\cdot)}\in LH_0\cap LH_{\infty}$ is equivalent to ${p(\cdot)}\in LH_0\cap LH_{\infty}$. Thus, the sense
of the above formulation with $\frac{1}{p(\cdot)}$ is that it allows to include the case of unbounded~$p(\cdot)$.
\end{remark}

Examples constructed in the works \cite{CUFN} and \cite{PR} show that the conditions on $p(\cdot)$ in Theorem \ref{sufc} cannot be improved in the pointwise sense.
Nevertheless, these conditions are not necessary. For example, it was shown by the author \cite{L1} that there exist discontinuous (both locally and at infinity)
exponents $p(\cdot)\in {\mathcal P}$.
In \cite{Ne1}, Nekvinda found an improvement of Theorem \ref{sufc}, where the $LH_{\infty}$ condition is replaced by a weaker integral condition defined as follows.

\begin{definition}\label{nekc}
We say that $p(\cdot)$ satisfies the $N_{\infty}$ condition if there exist $p_{\infty}>0$ and $0<c<1$ such that
$$\int_{{\mathbb R}^n}c^{\frac{1}{|p(x)-p_{\infty}|}}dx<\infty.$$
\end{definition}

Here, and throughout the paper, we adopt a convention that $\frac{1}{0}=\infty$ and $c^{\infty}=0$ for $0<c<1$.

It is easy to show that $LH_{\infty}\Rightarrow N_{\infty}$, but the converse is not true (see, e.g., \cite[Prop. 4.9]{CUF}). The following result is due to Nekvinda \cite{Ne1},
see also \cite[Th. 4.7]{CUF} and \cite[Rem. 4.3.10]{DHHR11}.

\begin{theorem}\label{sufcn}
If $p_->1$ and $\frac{1}{p(\cdot)}\in LH_0\cap N_{\infty}$, then $p(\cdot)\in {\mathcal P}$.
\end{theorem}

If $p(\cdot)$ is bounded, then, similarly to Remark \ref{rem}, the assumption $\frac{1}{p(\cdot)}\in LH_0\cap N_{\infty}$ can be written in the form $p(\cdot)\in LH_0\cap N_{\infty}$.

From now on assume that $p_->1$ and $p_+<\infty$. In this case the $LH_0$ condition in Theorem \ref{sufcn} can be replaced by a natural necessary condition of Muckenhoupt-type.
Given a $p(\cdot)$, define the conjugate exponent $p'(\cdot)$ by
$$p'(x):=\frac{p(x)}{p(x)-1}.$$

\begin{definition} We say that $p(\cdot)$ satisfies the $A_{p(\cdot)}$ condition if
$$\sup_{Q}\frac{\|\chi_Q\|_{L^{p(\cdot)}}\|\chi_Q\|_{L^{p'(\cdot)}}}{|Q|}<\infty,$$
where the supremum is taken over all cubes $Q\subset {\mathbb R}^n$.
\end{definition}

There are several different notations for this condition. For example, in \cite{CUF} it is denoted by $K_0$, and in \cite{DHHR11} it is denoted by ${\mathcal A}_{{\rm{loc}}}$.
Observe that for $L^{p(\cdot)}$ replaced by the weighted $L^p(w)$ space the $A_{p(\cdot)}$ condition is exactly the classical $A_p$ condition of Muchenhoupt (see, e.g., \cite[p.~143]{CUF}).

By duality between the $L^{p(\cdot)}$ and $L^{p'(\cdot)}$ spaces, the $A_{p(\cdot)}$ condition can be written in the equivalent form: there exists $C>0$ such that for every cube $Q$ and every locally integrable function $f$,
$$\langle |f|\rangle_Q\|\chi_Q\|_{L^{p(\cdot)}}\le C\|f\chi_Q\|_{L^{p(\cdot)}},$$
where we use a notation $\langle f\rangle_Q:=\frac{1}{|Q|}\int_Qf$. From this, using the fact that $\langle |f|\rangle_Q\chi_Q\le M(f\chi_Q)$, we obtain that the boundedness of $M$ on $L^{p(\cdot)}$ trivially implies the
$A_{p(\cdot)}$ condition. However, in contrast to the weighted $L^p(w)$ spaces, the $A_{p(\cdot)}$ condition is not sufficient for the boundedness of $M$ on $L^{p(\cdot)}$. The corresponding examples can be found in
\cite[Ex. 4.51]{CUF}, \cite[Th. 5.3.4]{DHHR11}, and \cite{K2}.

Although the $A_{p(\cdot)}$ condition is not sufficient, it plays a quite important role in further developments. In particular, as mentioned above, the $LH_0$ condition in Theorem \ref{sufcn} can be replaced by $A_{p(\cdot)}$, namely, we have the following result.

\begin{theorem}\label{sufconcr}
If $1<p_-\le p_+<\infty$ and $p(\cdot)\in A_{p(\cdot)}\cap N_{\infty}$, then $p(\cdot)\in {\mathcal P}$.
\end{theorem}

In the present form this result can be found in \cite[Th. 4.52]{CUF}. Closely related results in this direction were obtained in \cite{K1} and \cite{L2}.

By the above observations, $LH_0\cap N_{\infty}\Rightarrow {\mathcal P}\Rightarrow A_{p(\cdot)}$. This shows that the assumptions in Theorem \ref{sufconcr} are not stronger than the ones in Theorem \ref{sufcn}. In fact,
they are weaker, the corresponding example can be found in \cite[Ex. 4.59]{CUF}.

The only drawback in Theorem \ref{sufconcr} is a control of $p(\cdot)$ at infinity, expressed in the condition $N_{\infty}$. The examples showing that this condition is not necessary for
$p(\cdot)\in {\mathcal P}$ can be found in the works of Nekvinda \cite{Ne2, Ne3, Ne4}, see also \cite[Ex. 4.13]{CUF}. Observe that the works \cite{Ne3,Ne4} provide rather delicate sufficient conditions on $p(\cdot)$ at infinity. However, they are still not necessary. We will discuss them below in more detail in Section \ref{ss63}.

Turn now to Diening's characterization \cite{D1} mentioned above. Given a family of pairwise disjoint cubes ${\mathcal F}$, denote
$$A_{{\mathcal F}}f:=\sum_{Q\in {\mathcal F}}\langle f\rangle_Q\chi_Q.$$

\begin{theorem}\label{Di} Assume that $1<p_-\le p_+<\infty$. Then $p(\cdot)\in {\mathcal P}$ if and only if there exists a constant $C>0$ such that for every family of pairwise disjoint cubes ${\mathcal F}$,
$$\|A_{{\mathcal F}}f\|_{L^{p(\cdot)}}\le C\|f\|_{L^p(\cdot)}.$$
\end{theorem}

This theorem was proved by Diening \cite{D1}. Since $A_{{\mathcal F}}f\le Mf$, one direction in this theorem is obvious. However, the converse direction is rather non-trivial with a quite complicated proof.
Since the operator $A_{\mathcal F}$ is self-adjoint, an immediate corollary of this theorem is that for $p(\cdot)$ with $1<p_-\le p_+<\infty$ we have that $p(\cdot)\in {\mathcal P}$ if and only if $p'(\cdot)\in {\mathcal P}$.

While Theorem \ref{Di} is a characterization of the class ${\mathcal P}$, it still cannot be regarded as a full solution to Problem \ref{mq}. It is rather an important reduction of the initial problem to a simpler operator but not a characterization in terms of $p(\cdot)$.

Before we turn to our contribution, note that the exposition above is definitely not complete. In particular, let us also mention the works \cite{CCF,CUDF,DHHMS, KK,KZ} dedicated to various aspects of the class ${\mathcal P}$.

\subsection{Our contribution}
In our main result we replace the $N_{\infty}$ condition in Theorem \ref{sufconcr} by a weaker condition, which we call the ${\mathcal U}_{\infty}$ condition, such that $A_{p(\cdot)}\cap {\mathcal U}_{\infty}$ is necessary
and sufficient for $p(\cdot)\in {\mathcal P}$. The key role in our new condition is played by the following function.

\begin{definition}\label{fpla} Assume that $1<p_-\le p_+<\infty$. Given $\la,\tau\in (0,1)$, define the function $F_{p,\la,\tau}$ on ${\mathbb R}^{2n}$ by
$$F_{p,\la,\tau}(x,y):=\tau^{\frac{p(y)}{p(x)-p(y)}}\la^{\frac{p'(x)}{p'(y)-p'(x)}}\chi_{\{(x,y):p(x)>p(y)\}}.$$
\end{definition}

For computational purposes it is perhaps more convenient to write this definition in terms of $p(\cdot)$ only:
$$F_{p,\la,\tau}(x,y)=\big(\tau^{p(y)}\la^{p(x)(p(y)-1)}\big)^{\frac{1}{p(x)-p(y)}}\chi_{\{(x,y):p(x)>p(y)\}}.$$
However, the use of the conjugate function clarifies the following important symmetry property:
$$F_{p,\la,\tau}(x,y)=F_{p',\tau,\la}(y,x).$$

We are now ready to define the ${\mathcal U}_{\infty}$ condition.

\begin{definition}\label{uinf}
Assume that $1<p_-\le p_+<\infty$.
We say that $p(\cdot)$ satisfies the ${\mathcal U}_{\infty}$ condition if there exist $C>0, r>1$ and $0<\ga<1/2$ such that
for all $0<\la,\tau<\ga$ and for any finite family of pairwise disjoint cubes ${\mathcal F}$,
\begin{equation}\label{uinfty}
\sum_{Q\in {\mathcal F}}\frac{1}{|Q|}\inf_{E\subset Q:|E|=\la|Q|\atop G\subset Q:|G|=\tau|Q|}\int_{Q\setminus G}\int_{Q\setminus E}F_{p,\la,\tau}(x,y)^{1/r}dxdy\le C.
\end{equation}
\end{definition}

By the symmetry property mentioned above and since the roles of $\tau$ and $\la$ in Definition \ref{uinf} can be reversed, we have that $p(\cdot)\in {\mathcal U}_{\infty}$
if and only if $p'(\cdot)\in {\mathcal U}_{\infty}$.

From the practical point of view, a typical way of verifying the ${\mathcal U}_{\infty}$ condition is to show that there exists an integrable function $\psi\ge 0$ such that for every cube $Q$,
$$\inf_{E\subset Q:|E|=\la|Q|\atop G\subset Q:|G|=\tau|Q|}\int_{Q\setminus G}\int_{Q\setminus E}F_{p,\la,\tau}(x,y)^{1/r}dxdy\le C|Q|\int_Q\psi(x)dx,$$
uniformly in $0<\la,\tau<\ga$.

Observe that the ${\mathcal U}_{\infty}$ condition follows from $N_{\infty}$ trivially in the most rough way, namely, it is enough to estimate $\tau^{p(y)}\la^{p(x)(p(y)-1)}$
from the definition of $F_{p,\la,\tau}$ by $\ga^{p_-}\ga^{p_-(p_--1)}=\ga^{p_-^2}$ and to replace the double integral over $(Q\setminus G)\times (Q\setminus E)$ by the integral over $Q\times Q$,
see Section~\ref{ss61} for more detail. In more general situations the fact that we allow to remove subsets $E,G\subset Q$ is crucial. See Theorem \ref{ussc} below, where we obtain a useful and easy to check
condition implying that $p(\cdot)\in {\mathcal U}_{\infty}$.

Our main result is the following.

\begin{theorem}\label{mr} Assume that $1<p_-\le p_+<\infty$. Then $p(\cdot)\in {\mathcal P}$ if and only if $p(\cdot)\in A_{p(\cdot)}\cap {\mathcal U}_{\infty}$.
\end{theorem}

The proof of this result is based essentially on Theorem \ref{Di}, and our key tool is non-increasing rearrangements. In fact, we obtain a number of intermediate characterizations of the class ${\mathcal P}$
formulated in Section \ref{ss4}. They can be seen as a result of the process of passing from integrals to rearrangements.
What we are doing finally in the proof of Theorem \ref{mr} represents a converse process, that is, we come back from rearrangements to integrals. In Section \ref{ss2} we provide an informal discussion and
a sketch of the proof of Theorem \ref{mr}.

In conclusion we formulate several questions related to Theorem \ref{mr}.

\begin{que}\label{q1}
It would be interesting to extend Theorem \ref{mr} to the case of unbounded exponents. Indeed, the maximal operator is trivially bounded on $L^{\infty}$, and therefore this case is quite natural.
Theorems~\ref{sufc} and~\ref{sufcn} allow unbounded exponents. However, in our method of the proof the restriction $p_+<\infty$ is essential.
\end{que}

\begin{que}\label{q2}
It is crucial in the proof of Theorem \ref{mr} that the ${\mathcal U}_{\infty}$ condition is formulated with some $r>1$. On the other hand, this looks somewhat unnatural.
In analogy with the weighted $L^p(w)$ spaces, this looks like we prove the boundedness of $M$ on $L^p(w)$ assuming the $A_{p-\e}$ condition instead of $A_p$. Therefore, it is natural to ask whether
the ${\mathcal U}_{\infty}$ condition is self-improving, namely, whether (\ref{uinfty}) with $r=1$ (perhaps coupled with the $A_{p(\cdot)}$ condition) implies (\ref{uinfty}) with some $r>1$.
\end{que}

\begin{que}\label{g3} In our opinion, Theorem \ref{mr} is a satisfactory solution to Problem \ref{mq}. However, this does not exclude that another characterization of the class ${\mathcal P}$ could be found in simpler terms. In particular, it is not clear to us whether the ${\mathcal U}_{\infty}$ condition can be essentially simplified still being a necessary and sufficient condition (of course along with $A_{p(\cdot)}$).
Currently we have in this condition two independent parameters $\la$ and $\tau$. We do not know whether one can get rid of one of them.
\end{que}

The paper is organized as follows. In Section \ref{ss2} we discuss informally some ideas used in the proof of Theorem \ref{mr}.
Section \ref{ss3} is devoted to non-increasing rearrangements. In Section \ref{ss4} we obtain several intermediate characterizations of the class ${\mathcal P}$. In Section \ref{ss5} we prove Theorem~\ref{mr}.
Section~\ref{ss6} is devoted to the ${\mathcal U}_{\infty}$ condition. In particular, we give new examples of exponents $p(\cdot)\in {\mathcal P}$.

\section{A sketch of the proof}\label{ss2}
In this section we will briefly outline our approach to proving Theorem \ref{mr}. We split the proof into a number of intermediate steps, where each step represents its own characterization of the class ${\mathcal P}$.
Throughout the proof we deal with localized non-increasing rearrangements of the form $(p\chi_Q)^*(\la|Q|)$, where $Q\subset {\mathbb R}^n$ is a cube, and $0<\la<1$. We also assume that $p_->1$ and $p_+<\infty$,
and ${\mathcal F}$ is an arbitrary finite family of pairwise disjoint cubes.

\vskip 2mm
\noindent
{\bf Step 1.} We start by showing that $p(\cdot)\in {\mathcal P}$ if and only if $p(\cdot)\in A_{p(\cdot)}$ and the following property holds:
\begin{equation}\label{st1}
\sum_{Q\in {\mathcal F}}|Q|\int_0^{\la}\a_Q^{(p\chi_Q)^*(t|Q|)}dt\le 1\Rightarrow \sum_{Q\in {\mathcal F}}\int_Q(\la^{1/r}\a_Q)^{p(x)}dx\le C,
\end{equation}
where $C>0,r>1$, $0<\la<\la_0<1$, and $0<\a_Q<1$.
This is proved in Theorem \ref{char1} along with Remark \ref{eqform}.

Of course, this characterization is rather technical, and our further goal is to simplify it. More precisely, our primary goal is to get rid of the sequences $\{\a_Q\}_{Q\in {\mathcal F}}$ in (\ref{st1}).
The idea is to find an extremal sequence $\{\a_Q\}$ satisfying (\ref{st1}), and to insert it to the second sum in~(\ref{st1}). In such a way we would obtain a one sum condition, and without arbitrary~$\{\a_Q\}$.
This task is not so easy mainly because we have an integral of the form $\int_Qt^{p(x)}dx$ on the right-hand side of (\ref{st1}), which is not so convenient to deal with.
That is why in our second step we replace it (and the integral on the left-hand side as well) by an expression involving rearrangements.

\vskip 2mm
\noindent
{\bf Step 2.} We show in Theorem \ref{char2} that $p(\cdot)\in {\mathcal P}$ if and only if $p(\cdot)\in A_{p(\cdot)}$ and the following property holds:
$$
\la\sum_{Q\in {\mathcal F}}|Q|\a_Q^{(p\chi_Q)^*(\la|Q|)}\le 1\Rightarrow \tau^{1/r}\sum_{Q\in {\mathcal F}}|Q|(\la^{1/r}\a_Q)^{(p\chi_Q)^*((1-\tau)|Q|)}\le C,
$$
where, additionally to above parameters, we also have $0<\tau<\tau_0<1$.

Thus, we remove the integral on the right-hand side of (\ref{st1}) but the price is an additional
parameter $\tau$. In such a way two parameters $\la$ and $\tau$ appear in the ${\mathcal U}_{\infty}$ condition. Now, a simple analysis of power functions allows us to get rid of sequences~$\{\a_Q\}$.

\vskip 2mm
\noindent
{\bf Step 3.}
We show that $p(\cdot)\in {\mathcal P}$ if and only if $p(\cdot)\in A_{p(\cdot)}$ and the following property holds:
\begin{equation}\label{st3}
\sum_{Q\in {\mathcal F}}|Q|\tau^{\frac{1}{r}(1+\Psi_{Q,p}(\la,\tau))}\la^{\frac{1}{r}(1+\Psi_{Q,p'}(\tau,\la))}\le C,
\end{equation}
where
$$
\Psi_{Q,p}(\la,\tau):=\frac{(p\chi_Q)^*((1-\tau)|Q|)}{(p\chi_Q)^*(\la|Q|)-(p\chi_Q)^*((1-\tau)|Q|)},
$$
and $0<\la,\tau<\ga_0<1/2$. This is done in Lemma \ref{aneq} and Theorem~\ref{fstm}.

Observe that
$$\lim_{t\to 0}(p\chi_Q)^*(t|Q|)=\esssup_{Q}p\quad\text{and}\quad \lim_{t\to 1}(p\chi_Q)^*(t|Q|)=\essinf_Qp.$$
Thus, if $\ga_0$ is small enough, and $0<\la,\tau<\ga_0$ we have that $(p\chi_Q)^*(\la|Q|)$ is close to the supremum of $p(\cdot)$ on $Q$ and $(p\chi_Q)^*((1-\tau)|Q|)$ is close to the infimum of $p(\cdot)$ on $Q$.
For this reason, (\ref{st3}) can be thought as an oscillation property of $p(\cdot)$.

Our next goal is to rewrite (\ref{st3}) in the language of integrals. This is a technically difficult task. First we observe that $(\la\tau)^{1/r}$ in (\ref{st3}) can be replaced by $\la\tau$.

\vskip 2mm
\noindent
{\bf Step 4.} We show in Theorem \ref{fstm1} that $p(\cdot)\in {\mathcal P}$ if and only if $p(\cdot)\in A_{p(\cdot)}$ and the following property holds:
\begin{equation}\label{st4}
\sum_{Q\in {\mathcal F}}|Q|(\tau\lambda)\tau^{\frac{1}{r}\Psi_{Q,p}(\la,\tau)}\la^{\frac{1}{r}\Psi_{Q,p'}(\tau,\la)}\le C.
\end{equation}

We emphasize that all the above statements are understood in the sense ``there exists $r>1$ such that...". In particular, $r>1$ may differ from one statement to another.
The direction that $(\ref{st3})\Rightarrow (\ref{st4})$ is obvious with the same $r>1$. In the converse direction, we show that given an $r>1$ for which (\ref{st4}) holds, there exists $1<s<r$ such that (\ref{st3}) holds
with $s$ instead of $r$.

An interplay between (\ref{st3}) and (\ref{st4}) allows us to establish the following intermediate characterization.

\vskip 2mm
\noindent
{\bf Step 5.} We show in Theorem \ref{yac} that $p(\cdot)\in {\mathcal P}$ if and only if $p(\cdot)\in A_{p(\cdot)}$ and the following property holds:
\begin{equation}\label{st5}
\sum_{Q\in {\mathcal F}}|Q|\int_{\la}^{\ga_0}\int_{\tau}^{\ga_0}
\tau^{\frac{1}{r}\Psi_{Q,p}(t,s)}\la^{\frac{1}{r}\Psi_{Q,p'}(s,t)}ds dt\le C.
\end{equation}

Finally, we may connect (\ref{st5}) with the function $F_{p,\la,\tau}$ from Definition~\ref{fpla}.

\vskip 2mm
\noindent
{\bf Step 6.} We show in Lemma \ref{computr} that for almost all $t,s\in (0,1)$ such that $t+s<1$,
\begin{equation}\label{st6}
(F_{p,\la,\tau}\chi_{Q\times Q})^*(t|Q|,s|Q|)=\tau^{\Psi_{Q,p}(t,s)}\la^{\Psi_{Q,p'}(s,t)}.
\end{equation}

The object on the left-hand side here is the iterated rearrangement, see Section \ref{ss32} for its definition.
This is the way on how the function $F_{p,\la,\tau}$ comes to the picture. Now, a combination of (\ref{st5}) and (\ref{st6}) leads
to the proof of Theorem~\ref{mr} (although, there is still some technical work which we will not mention here).

\section{Non-increasing rearrangements}\label{ss3}
\subsection{One-variable rearrangements}\label{ss31}
Given a measurable function $f$ on ${\mathbb R}^n$, its non-increasing rearrangement $f^*$ is defined by
$$f^*(t):=\inf\{\a>0:|\{x\in {\mathbb R}^n:|f(x)|>\a\}|\le t\}\quad(0<t<\infty).$$
If $f$ is defined on a set $E\subset {\mathbb R}^n$, its rearrangement is understood as $f^*(t):=(f\chi_E)^*(t)$.

For main properties of the rearrangements we refer to \cite[Ch. 2]{BS}. Here we mention several specific facts which will be frequently used
in this paper.


\begin{prop}\label{fql} Let $f$ be a locally integrable function on ${\mathbb R}^n$. Then, for every cube $Q\subset {\mathbb R}^n$ and for all $\la\in (0,1)$,
\begin{equation}\label{rearpr}
\int_Q|f(x)|dx\le \frac{1}{\la}\int_0^{\la|Q|}(f\chi_Q)^*(t)dt
\end{equation}
and
\begin{equation}\label{infpr}
\inf_{E\subset Q:|E|=\la|Q|}\int_E|f(x)|dx=\int_{(1-\la)|Q|}^{|Q|}(f\chi_Q)^*(t)dt.
\end{equation}
\end{prop}

\begin{proof} To show (\ref{rearpr}), it is enough to note that
$$\int_Q|f(x)|dx=\int_0^{|Q|}(f\chi_Q)^*(t)dt$$
and to use
the fact that the function $f^{**}(t):=\frac{1}{t}\int_0^{t}f^*(s)ds$ is non-increasing (see, e.g., \cite[Pr. 3.2]{BS}).

Now, (\ref{infpr}) is an equivalent form of the well-known property
\begin{equation}\label{rearpr1}
\sup_{E\subset Q:|E|=s|Q|}\int_E|f(x)|dx=\int_0^{s|Q|}(f\chi_Q)^*(t)dt\quad(0<s<1)
\end{equation}
(see, e.g., \cite[Pr. 3.3]{BS}), taking into account that
$$
\inf_{E\subset Q:|E|=\la|Q|}\int_E|f|=\int_Q|f|-\sup_{G\subset Q:|G|=(1-\la)|Q|}\int_G|f|.
$$
\end{proof}

\begin{remark}\label{bsh}
Observe that the supremum in (\ref{rearpr1}) is attained on an arbitrary subset
$$E\subset \{x\in Q:|f(x)|\ge (f\chi_Q)^*(s|Q|)\}$$ of measure $s|Q|$
(see, e.g, \cite[Lemma 2.5]{BS}). Therefore, the infimum in (\ref{infpr}) is attained on an arbitrary subset
$$E\subset \{x\in Q:|f(x)|\le (f\chi_Q)^*((1-\la)|Q|)\}$$
of measure $\la|Q|$.
\end{remark}

Several statements below are about how the rearrangements commute with compositions.
They are apparently well known but we could not find them in the literature in the form presented below.

\begin{lemma}\label{compos} Let $f$ be a non-negative measurable function on a set $E\subset {\mathbb R}^n$ of finite measure.
\begin{enumerate}[(i)]
\item
Assume that $\f$ is a non-negative, strongly increasing continuous function on an interval which contains $f(E)$. Then for $\psi:=\f\circ f$ and for all $t\in (0,|E|)$ we have
$$\psi^*(t)=\f\big(f^*(t)\big).$$
\item
Assume that $\f$ is a non-negative, strongly decreasing continuous function on an interval which contains $f(E)$. Then for $\psi:=\f\circ f$ we have
$$\psi^*(t)=\f\big(f^*(|E|-t)\big)$$
for almost all $t\in (0,|E|)$.
\end{enumerate}
\end{lemma}

\begin{proof} We start with part (i). Using that
$$\{x\in E:\psi(x)>\a\}=\{x\in E:f(x)>\f^{-1}(\a)\},$$
we obtain
$$|\{x\in E: \psi(x)>\f\big(f^*(t)\big)\}|\le t,$$
and, for any $\e>0$,
$$|\{x\in E: \psi(x)>\f\big(f^*(t)\big)-\e\}|> t,$$
from which we conclude that $\psi^*(t)=\f\big(f^*(t)\big)$.

Turn to part (ii). We have
\begin{eqnarray}
|\{x\in E:\psi(x)>\a\}|&=&|\{x\in E:f(x)<\f^{-1}(\a)\}|\nonumber\\
&=&|E|-|\{x\in E:f(x)\ge \f^{-1}(\a)\}|.\label{sesti}
\end{eqnarray}
From this,
$$|\{x\in E: \psi(x)>\f\big(f^*(|E|-t)\big)\}|\le t,$$
and hence
\begin{equation}\label{fpart}
\psi^*(t)\le\f\big(f^*(|E|-t)\big).
\end{equation}

On the other hand, using (\ref{sesti}), for $\e>0$ we obtain
$$|\{x\in E: \psi(x)>\f\big(f^*(|E|-t)\big)-\e\}|\ge t.$$
From this, for any $\d>0$,
$$
\f\big(f^*(|E|-t)\big)-\e\le \psi^*(t-\d).
$$
Hence, assuming that $t$ is the point of continuity of $\psi^*(t)$ and letting $\e,\d\to 0$, we obtain
$$\f\big(f^*(|E|-t)\big)\le \psi^*(t),$$
which, along with (\ref{fpart}), completes the proof.
\end{proof}

\begin{remark}\label{explan} A simple example shows that  the requirement ``almost all" in the second part of Lemma \ref{compos} cannot be replaced by ``all".
Indeed, assume that
$$
f(x):=
\begin{cases}
1,&x\in (0,1/2]\\
2,&x\in (1/2,1)
\end{cases}
$$
and $\f(x):=\frac{1}{x}$. Then for $t\in (0,1)$ we have
$$
f^*(t)=
\begin{cases}
2,&t\in (0,1/2)\\
1,&t\in [1/2,1),
\end{cases}
$$
and, for $\psi(x):=\frac{1}{f(x)}$,
$$
\psi^*(t)=
\begin{cases}
1,&t\in (0,1/2)\\
1/2,&t\in [1/2,1).
\end{cases}
$$
Therefore, the equality $\psi^*(t)=\frac{1}{f^*(1-t)}$ does not hold for $t=1/2$.
\end{remark}

\begin{lemma}\label{fila} Let $f$ be a non-negative measurable function on a set $E\subset {\mathbb R}^n$, and let $\a\ge 0$.
Then, for
$$g(x):=f(x)\chi_{\{x\in E:f(x)>\a\}}(x)$$
we have
$$g^*(t)=f^*(t)\chi_{\{t\in (0,|E|):f^*(t)>\a\}}(t).$$
\end{lemma}

\begin{proof}
Denote $E_{\a}:=\{x\in E:f(x)>\a\}$. Suppose that $f^*(t)\le \a$. Then $|E_{\a}|\le t$. Since $g$ is supported in $E_{\a}$, we obtain that $g^*(t)=0$.

Suppose now that $f^*(t)>\a$. Then, on the one hand, we have
$$|\{x\in E:g(x)>f^*(t)\}|=|\{x\in E_{\a}: f(x)>f^*(t)\}|\le t.$$
On the other hand, for $\e>0$,
\begin{eqnarray*}
&&|\{x\in E:g(x)>f^*(t)-\e\}|=|\{x\in E_{\a}:f(x)>f^*(t)-\e\}|\\
&&=|\{x\in E:f(x)>\max(\a, f^*(t)-\e)\}|>t,
\end{eqnarray*}
which proves that $g^*(t)=f^*(t)$.
\end{proof}

\begin{lemma}\label{comp1}
Let $f$ be a non-negative measurable function on a set $E\subset {\mathbb R}^n$ of finite measure, and let $\a>0$.
\begin{enumerate}[(i)]
\item
Assume that $\f$ is a non-negative, strongly increasing continuous function on $[\a,\sup_Ef]$. Then for
$$g(x):=\f(f(x))\chi_{\{x\in E:f(x)>\a\}}(x)$$
we have
$$g^*(t)=\f(f^*(t))\chi_{\{t\in (0,|E|):f^*(t)>\a\}}(t).$$
\item
Assume that $\f$ is a non-negative, strongly decreasing continuous function on $[\inf_Ef,\a]$. Then for
$$g(x):=\f(f(x))\chi_{\{x\in E:f(x)<\a\}}(x)$$
we have
$$g^*(t)=\f\big(f^*(|E|-t)\big)\chi_{\{t\in (0,|E|):f^*(|E|-t)<\a\}}(t)$$
for almost all $t\in (0,|E|)$.
\end{enumerate}
\end{lemma}

\begin{proof}
Part (i) is a combination of the first part of Lemma \ref{compos} along with Lemma \ref{fila} and the representations
$$\{x\in E:f(x)>\a\}=\{x\in E:\f(f(x))>\f(\a)\}$$
and
$$\{t\in (0,|E|):\f(f^*(t))>\f(\a)\}=\{t\in (0,|E|):f^*(t)>\a\}.$$

Part (ii) follows in the same way with the use of the second part of Lemma \ref{compos} along with Lemma \ref{fila} and the representations
$$\{x\in E:f(x)<\a\}=\{x\in E:\f(f(x))>\f(\a)\}$$
and
$$\{t\in (0,|E|):\f\big(f^*(|E|-t)\big)>\f(\a)\}=\{t\in (0,|E|):f^*(|E|-t)<\a\}.$$
\end{proof}

\subsection{Iterated rearrangements}\label{ss32}
Assume that $f(x,y)$ is a measurable function on ${\mathbb R}^{2n}$, where $x,y\in {\mathbb R}^n$. For a fixed $y$, define the usual rearrangement of $f(x,\cdot)$ with respect to $x$, and denote the resulting function by
$f^*(t,y), t>0, y\in {\mathbb R}^n$. Now, for a fixed $t$, define the usual rearrangement of $f^*(\cdot,y)$ with respect to $y$, and denote the resulting function by $f^*(t,s), t,s>0$.
Observe that this object is well known, see, e.g., \cite{Bl} for its basic properties.

We will need the following analogue of (\ref{infpr}).

\begin{prop}\label{itrpr} Let $f$ be a locally integrable function on ${\mathbb R}^{2n}$. Then, for every cube of the form $Q\times Q\subset {\mathbb R}^{2n}$ and for all $\la,\tau\in (0,1)$,
$$
\int_{(1-\la)|Q|}^{|Q|}\int_{(1-\tau)|Q|}^{|Q|}(f\chi_{Q\times Q})^*(t,s)dsdt\le \inf_{E\subset Q:|E|=\la|Q| \atop G\subset Q:|G|=\tau|Q|}\int_E\int_G|f(x,y)|dydx.
$$
\end{prop}

\begin{proof} This is just a combination of (\ref{infpr}) with Fubini's theorem:
\begin{eqnarray*}
&&\int_{(1-\la)|Q|}^{|Q|}\int_{(1-\tau)|Q|}^{|Q|}(f\chi_{Q\times Q})^*(t,s)dsdt\\
&&=\int_{(1-\la)|Q|}^{|Q|}\Big(\inf_{G\subset Q:|G|=\tau|Q|}\int_G(f\chi_{Q\times Q})^*(t,y)dy\Big)dt\\
&&\le \inf_{G\subset Q:|G|=\tau|Q|}\int_{(1-\la)|Q|}^{|Q|}\int_G(f\chi_{Q\times Q})^*(t,y)dydt\\
&&=\inf_{G\subset Q:|G|=\tau|Q|}\int_G\int_{(1-\la)|Q|}^{|Q|}(f\chi_{Q\times Q})^*(t,y)dtdy\\
&&=\inf_{G\subset Q:|G|=\tau|Q|}\int_G\Big(\inf_{E\subset Q:|E|=\la|Q|}\int_E|f(x,y)|dx\Big)dy\\
&&\le \inf_{E\subset Q:|E|=\la|Q| \atop G\subset Q:|G|=\tau|Q|}\int_E\int_G|f(x,y)|dydx.
\end{eqnarray*}
\end{proof}

\begin{remark}\label{mbstr} Observe that, contrary to (\ref{infpr}), we have an inequality in Proposition \ref{itrpr}. A simple example shows that it can be strict. Indeed, for $(x,y)\in {\mathbb R}^2$
define $f(x,y):=\chi_E(x,y)$, where $E$ is the union of two squares $(0,1)\times (0,1)$ and $(1,2)\times (1,2)$. Then, for all $y\in (0,1)\cup (1,2)$,
$$
f^*(t,y)=
\begin{cases}
1,&t\in (0,1)\\
0,&t\in [1,2),
\end{cases}
$$
and, therefore,
$$
f^*(t,s)=
\begin{cases}
1,&t\in (0,1), s\in(0,2)\\
0,&t\in [1,2), s\in(0,2).
\end{cases}
$$
\end{remark}
From this, for all $\la\in (0,1/2]$ and $\tau\in (0,1)$,
$$\int_{2(1-\la)}^2\int_{2(1-\tau)}^2(f\chi_{(0,2)\times(0,2)})^*(t,s)dsdt=0.$$
However, if $\tau\in (1/2,1)$, then for all $\la\in (0,1)$,
$$
\inf_{E\subset (0,2):|E|=2\la \atop G\subset (0,2):|G|=2\tau}\int_E\int_G|f(x,y)|dydx>0.
$$

\section{Intermediate characterizations of the class ${\mathcal P}$}\label{ss4}
In this section we obtain several characterizations of the class ${\mathcal P}$. They will play an important role in the proof of Theorem \ref{mr}, and also they are of some intrinsic interest.

\subsection{The $S_{\infty}$ condition}\label{ss41}
The following definition looks rather long but it expresses the key property of $p(\cdot)$ needed to characterize the class~${\mathcal P}$.

\begin{definition}\label{cond}
We say that $p(\cdot)$ satisfies the $S_{\infty}$ condition if there exist $r>1$, $0<\la_0<1$ and $C>0$ such that for every $0<\la<\la_0$, for any finite family of pairwise disjoint cubes ${\mathcal F}$
and for every sequence $\{\a_Q\}_{Q\in {\mathcal F}}$ with $0<\a_Q<1$ the following property holds: for any measurable subsets $E_Q\subset Q$ with $|E_Q|=\la|Q|$, $Q\in {\mathcal F}$, we have
\begin{equation}\label{scond}
\sum_{Q\in {\mathcal F}}\int_{E_Q}\a_Q^{p(x)}dx\le 1\Rightarrow \sum_{Q\in {\mathcal F}}\int_Q(\la^{1/r}\a_Q)^{p(x)}dx\le C.
\end{equation}
\end{definition}

In the proof of the next theorem we will use the following duality relation (see \cite[Th. 3.2.13]{DHHR11}): if $1<p_-\le p_+<\infty$, then
\begin{equation}\label{dualr}
\frac{1}{2}\|f\|_{L^{p(\cdot)}}\le \sup_{g:\|g\|_{L^{p'(\cdot)}}=1}\int_{{\mathbb R}^n}|f(x)g(x)|dx\le 2\|f\|_{L^{p(\cdot)}}.
\end{equation}

We will also use the fact that if $M$ is bounded on $L^{p(\cdot)}$, then there exists $r>1$ such that the operator $M_r$ defined by $M_rf:=M(|f|^r)^{1/r}$ is also bounded on $L^{p(\cdot)}$.
In the context of variable $L^p$ spaces this was shown by Diening \cite{D1}. A different proof, in a more general context of Banach function spaces, can be found in \cite{LP}.

\begin{theorem}\label{char1} Assume that $1<p_-\le p_+<\infty$. Then $p(\cdot)\in {\mathcal P}$ if and only if $p(\cdot)\in A_{p(\cdot)}\cap S_{\infty}$.
\end{theorem}

\begin{proof} Let us start with a simple direction ${\mathcal P}\Rightarrow A_{p(\cdot)}\cap S_{\infty}$. It was discussed in the Introduction that ${\mathcal P}\Rightarrow A_{p(\cdot)}$. Let us show that
${\mathcal P}\Rightarrow~S_{\infty}$.

Assume that $p(\cdot)\in {\mathcal P}$. Then there exists $r>1$ such that $M_r$ is bounded on $L^{p(\cdot)}$. From this,
for every finite family ${\mathcal F}$ of pairwise disjoint cubes we obtain
$$\Big\|\sum_{Q\in {\mathcal F}}\Big(\frac{1}{|Q|}\int_Q|f|^r\Big)^{1/r}\chi_Q\Big\|_{L^{p(\cdot)}}\le C\|f\|_{L^{p(\cdot)}}.$$
Setting here $f:=\sum_{Q\in {\mathcal F}}\a_Q\chi_{E_Q}$, where $0<\a_Q<1$ and $E_Q\subset Q$ with $|E_Q|=\la|Q|$, yields
$$\Big\|\sum_{Q\in {\mathcal F}}(\la^{1/r}\a_Q)\chi_Q\Big\|_{L^{p(\cdot)}}\le C\Big\|\sum_{Q\in {\mathcal F}}\a_Q\chi_{E_Q}\Big\|_{L^{p(\cdot)}}.$$
From this,
$$\Big\|\sum_{Q\in {\mathcal F}}\a_Q\chi_{E_Q}\Big\|_{L^{p(\cdot)}}\le 1\Rightarrow \Big\|\sum_{Q\in {\mathcal F}}(\la^{1/r}\a_Q)\chi_Q\Big\|_{L^{p(\cdot)}}\le C,$$
which, by the definition of the $L^{p(\cdot)}$-norm, is equivalent to (\ref{scond}).

Turn to the converse direction. Assume that $p(\cdot)\in A_{p(\cdot)}\cap S_{\infty}$. In order to show that $p(\cdot)\in {\mathcal P}$, by Diening's Theorem \ref{Di} and by the standard limiting argument, it suffices to prove that
for every finite family of pairwise disjoint cubes ${\mathcal F}$ and for every $f\ge 0$,
$$\Big\|\sum_{Q\in {\mathcal F}}\langle f\rangle_Q\chi_Q\Big\|_{L^p(\cdot)}\le C\|f\|_{L^p(\cdot)},$$
which, in turn, is equivalent to
\begin{equation}\label{suftop}
\|f\|_{L^p(\cdot)}\le 1\Rightarrow \Big\|\sum_{Q\in {\mathcal F}}\langle f\rangle_Q\chi_Q\Big\|_{L^p(\cdot)}\le C.
\end{equation}

Fix an $f$ with $\|f\|_{L^p(\cdot)}\le 1$. Split $f$ in the standard way, $f=f_1+f_2$, where $f_1:=f\chi_{\{f\ge 1\}}$.
For the big part $f_1$ we use a trivial estimate
$$\Big\|\sum_{Q\in {\mathcal F}}\langle f_1\rangle_Q\chi_Q\Big\|_{L^p(\cdot)}\le \|M(f_1)\|_{L^p(\cdot)},$$
and we will show that
\begin{equation}\label{wws}
\|M(f_1)\|_{L^p(\cdot)}\le C.
\end{equation}

In order to handle the left hand of (\ref{wws}), we invoke the following fact proved in \cite{L2} (see \cite[Rem. 3.3]{L2}): if $p(\cdot)\in A_{p(\cdot)}$, then for every measurable set $E\subset {\mathbb R}^n$ of finite measure,
\begin{equation}\label{tams}
\|(Mf)\chi_E\|_{L^{p(\cdot)}}\le
c_{p(\cdot),n}c(E)\|f\|_{L^{p(\cdot)}},
\end{equation}
where $c(E):=\max\big(1,\|\chi_E\|_{L^{p(\cdot)}}^{1/(p'(\cdot))_-}\big)$.

First, we have
\begin{equation}\label{fwh}
\|M(f_1)\|_{L^{p(\cdot)}}\le \|M(f_1)\chi_{\{Mf_1>1\}}\|_{L^{p(\cdot)}}+
\|M(f_1)\chi_{\{Mf_1\le 1\}}\|_{L^{p(\cdot)}}.
\end{equation}
By the weak type $(1,1)$ of $M$ and using the fact that either $f_1=0$ or $f_1\ge 1$, we obtain
$$|\{Mf_1>1\}|\le C\|f_1\|_{L^1}\le C\int_{{\mathbb R}^n}f(x)^{p(x)}dx\le C.$$
Hence, $C(E)\le C$, where $E=\{Mf_1>1\}$, and by (\ref{tams}),
\begin{equation}\label{fmp}
\|M(f_1)\chi_{\{Mf_1>1\}}\|_{L^{p(\cdot)}}\le C\|f_1\|_{L^{p(\cdot)}}\le C.
\end{equation}

Turn to the second term in (\ref{fwh}). Using the $L^p$ boundedness of $M$ for $p>1$, we obtain
\begin{eqnarray*}
\int_{\{Mf_1\le 1\}}(Mf_1)^{p(x)}dx&\le& \int_{\{Mf_1\le
1\}}(Mf_1)^{p_-}dx\\
&\le& C\int_{{\mathbb R}^n}f_1^{p_-}dx\le C\int_{{\mathbb
R}^n}f^{p(x)}dx\le C.
\end{eqnarray*}
Therefore,
$$\|M(f_1)\chi_{\{Mf_1\le 1\}}\|_{L^{p(\cdot)}}\le C,$$
which, along with (\ref{fmp}), proves (\ref{wws}).

In order to establish (\ref{suftop}), it remains to show that
\begin{equation}\label{itrts}
\Big\|\sum_{Q\in {\mathcal F}}\langle f_2\rangle_Q\chi_Q\Big\|_{L^p(\cdot)}\le C.
\end{equation}

Given a cube $Q\in {\mathcal F}$, set $\a_Q:=(f_2\chi_Q)^*(\la|Q|)$, and also choose a subset
$$E_Q\subset \{x\in Q:f_2(x)\ge (f_2\chi_Q)^*(\la|Q|)\}$$
such that $|E_Q|=\la|Q|$. Then
$$\sum_{Q\in {\mathcal F}}\int_{E_Q}\a_Q^{p(x)}dx\le \sum_{Q\in {\mathcal F}}\int_Q(f_2)^{p(x)}dx\le \int_{{\mathbb R}^n}f(x)^{p(x)}dx\le 1.$$
Therefore, by (\ref{scond}), we obtain that
$$\sum_{Q\in {\mathcal F}}\int_Q\big(\la^{1/r}(f_2\chi_Q)^*(\la|Q|)\big)^{p(x)}dx\le C,$$
which is equivalent to
$$\Big\|\sum_{Q\in {\mathcal F}}(f_2\chi_Q)^*(\la|Q|)\chi_Q\Big\|_{L^{p(\cdot)}}\le \frac{C}{\la^{1/r}}\quad(0<\la<\la_0).$$

From this, by (\ref{dualr}), for every $g$ with $\|g\|_{L^{p'(\cdot)}}=1$,
$$\sum_{Q\in {\mathcal F}}(f_2\chi_Q)^*(\la|Q|)\int_Q|g|\le \frac{C}{\la^{1/r}}.$$
Integrating both sides over $\la\in (0,\la_0)$ and using (\ref{rearpr}) yield
$$\sum_{Q\in {\mathcal F}}\langle f_2\rangle_Q\int_Q|g|\le C.$$
Applying (\ref{dualr}) again, we obtain (\ref{itrts}), and therefore the proof is complete.
\end{proof}



\begin{remark}\label{eqform} Since (\ref{scond}) holds for any  measurable subsets $E_Q\subset Q$ with $|E_Q|=\la|Q|$, the $S_{\infty}$ condition can be written in the following equivalent form. Namely, under the same assumptions, (\ref{scond}) can be replaced by
$$
\sum_{Q\in {\mathcal F}}\inf_{E\subset Q:|E|=\la|Q|}\int_{E}\a_Q^{p(x)}dx\le 1\Rightarrow \sum_{Q\in {\mathcal F}}\int_Q(\la^{1/r}\a_Q)^{p(x)}dx\le C.
$$
However, the left-hand side here can be computed precisely. Indeed, since $0<\a_Q<1$, the function $\f(t):=\a_Q^t$ is decreasing, and hence, combining (\ref{infpr}) with the second part of Lemma \ref{compos} yields
\begin{eqnarray*}
&&\inf_{E\subset Q:|E|=\la|Q|}\int_{E}\a_Q^{p(x)}dx=\int_{(1-\la)|Q|}^{|Q|}(\a_Q^{p(\cdot)}\chi_Q)^*(s)ds\\
&&=\int_{(1-\la)|Q|}^{|Q|}\a_Q^{(p\chi_Q)^*(|Q|-s)}ds=|Q|\int_0^{\la}\a_Q^{(p\chi_Q)^*(t|Q|)}dt.
\end{eqnarray*}

Therefore, the $S_{\infty}$ condition is equivalent to saying that
\begin{equation}\label{eqbs}
\sum_{Q\in {\mathcal F}}|Q|\int_0^{\la}\a_Q^{(p\chi_Q)^*(t|Q|)}dt\le 1\Rightarrow \sum_{Q\in {\mathcal F}}\int_Q(\la^{1/r}\a_Q)^{p(x)}dx\le C.
\end{equation}
\end{remark}

\subsection{The $S_{\infty}'$ condition}\label{ss42}
Our next goal is to simplify the previous characterization by getting rid of sequences $\{\a_Q\}$. We will do it in two steps. In the first step we replace the integrals on both sides of (\ref{eqbs})
by expressions involving rearrangements.

\begin{definition}\label{cond1}
We say that $p(\cdot)$ satisfies the $S'_{\infty}$ condition if there exist $r>1$, $0<\la_0,\tau_0<1$ and $C>0$ such that for every $0<\la<\la_0$ and $0<\tau<\tau_0$, for any finite family of pairwise disjoint cubes ${\mathcal F}$
and for every sequence $\{\a_Q\}_{Q\in {\mathcal F}}$ with $0<\a_Q<1$ the following property holds:
$$
\la\sum_{Q\in {\mathcal F}}|Q|\a_Q^{(p\chi_Q)^*(\la|Q|)}\le 1\Rightarrow \tau^{1/r}\sum_{Q\in {\mathcal F}}|Q|(\la^{1/r}\a_Q)^{(p\chi_Q)^*((1-\tau)|Q|)}\le C.
$$
\end{definition}

The proof of the next result will be based on the following theorem obtained by Diening \cite{D1}.

\begin{theorem}\label{thd} Assume that $1<p_-\le p_+<\infty$, and let $p(\cdot)\in {\mathcal P}$. Then there exist $q>1$ and $C>0$ such that for every family of pairwise disjoint cubes ${\mathcal F}$ and for
every non-negative sequence $\{t_Q\}_{Q\in {\mathcal F}}$,
$$\sum_{Q\in {\mathcal F}}\int_Qt_Q^{p(x)}dx\le 1\Rightarrow \sum_{Q\in {\mathcal F}}|Q|\left(\frac{1}{|Q|}\int_Qt_Q^{qp(x)}dx\right)^{1/q}\le C.$$
\end{theorem}

\begin{theorem}\label{char2} Assume that $1<p_-\le p_+<\infty$. Then $p(\cdot)\in {\mathcal P}$ if and only if $p(\cdot)\in A_{p(\cdot)}\cap S'_{\infty}$.
\end{theorem}

\begin{proof} Assume that $p(\cdot)\in {\mathcal P}$. Then $p(\cdot)\in A_{p(\cdot)}$. Let us show that $p(\cdot)\in S'_{\infty}$.

By Theorem \ref{cond} combined with Remark \ref{eqform} we have that (\ref{eqbs}) holds.
Observing that
$$\int_0^{\la}\a_Q^{(p\chi_Q)^*(t|Q|)}dt\le \la\a_Q^{(p\chi_Q)^*(\la|Q|)},$$
we obtain
\begin{equation}\label{prom}
\la\sum_{Q\in {\mathcal F}}|Q|\a_Q^{(p\chi_Q)^*(\la|Q|)}\le 1\Rightarrow \sum_{Q\in {\mathcal F}}\int_Q(\la^{1/r}\a_Q)^{p(x)}dx\le C.
\end{equation}

By Theorem \ref{thd}, there exists $q>1$ such that the right-hand side of (\ref{prom}) implies
\begin{equation}\label{rhsim}
\sum_{Q\in {\mathcal F}}|Q|\left(\frac{1}{|Q|}\int_Q(\la^{1/r}\a_Q)^{qp(x)}dx\right)^{1/q}\le C.
\end{equation}
Setting $g(x):=(\la^{1/r}\a_Q)^{p(x)}$, by Chebyshev's inequality, we obtain
$$\tau^{1/q}(g\chi_Q)^*(\tau|Q|)\le \left(\frac{1}{|Q|}\int_Qg(x)^qdx\right)^{1/q}.$$
Therefore, by the second part of Lemma \ref{compos} along with (\ref{rhsim}), for almost all $\tau\in (0,1)$,
$$
\tau^{1/q}\sum_{Q\in {\mathcal F}}|Q|(\la^{1/r}\a_Q)^{(p\chi_Q)^*((1-\tau)|Q|)}\le C.
$$
From this, for $\nu:=\min(q,r)$,
\begin{equation}\label{anin}
\tau^{1/\nu}\sum_{Q\in {\mathcal F}}|Q|(\la^{1/\nu}\a_Q)^{(p\chi_Q)^*((1-\tau)|Q|)}\le C,
\end{equation}
for almost all $\tau\in (0,1)$.

Now observe that (\ref{anin}) is easily extended to all $\tau\in (0,1)$. Indeed, let $E\subset (0,1)$ be set of full measure for which (\ref{anin}) is already established.
Assume that $\tau_0\not\in E$. Then there exists
a sequence $\e_k\to 0$ as $k\to \infty$ such that $\tau_0-\e_k\in E$. We obtain that
\begin{eqnarray*}
&&(\tau_0-\e_k)^{1/\nu}\sum_{Q\in {\mathcal F}}|Q|(\la^{1/\nu}\a_Q)^{(p\chi_Q)^*((1-\tau_0)|Q|)}\\
&&\le (\tau_0-\e_k)^{1/\nu}\sum_{Q\in {\mathcal F}}|Q|(\la^{1/\nu}\a_Q)^{(p\chi_Q)^*((1-\tau_0+\e_k)|Q|)}\le C.
\end{eqnarray*}
Letting here $k\to\infty$ yields (\ref{anin}) for $\tau=\tau_0$. Therefore, (\ref{anin}) holds for all $\tau\in (0,1)$.

Hence, (\ref{prom}) implies that
$$\la\sum_{Q\in {\mathcal F}}|Q|\a_Q^{(p\chi_Q)^*(\la|Q|)}\le 1\Rightarrow
\tau^{1/\nu}\sum_{Q\in {\mathcal F}}|Q|(\la^{1/\nu}\a_Q)^{(p\chi_Q)^*((1-\tau)|Q|)}\le C,$$
and thus we have obtained the $S_{\infty}'$ condition.

Suppose now that $p(\cdot)\in A_{p(\cdot)}\cap S'_{\infty}$. Integrating both sides of
$$
\sum_{Q\in {\mathcal F}}|Q|(\la^{1/r}\a_Q)^{(p\chi_Q)^*((1-\tau)|Q|)}\le \frac{C}{\tau^{1/r}}
$$
over $\tau\in (0,\tau_0)$ and using (\ref{rearpr}) yield
$$
\sum_{Q\in {\mathcal F}}\int_Q(\la^{1/r}\a_Q)^{p(x)}dx\le C.
$$
Therefore, the $S_{\infty}'$ condition implies (\ref{prom}). Now observe that
$$\frac{\la}{2}\a_Q^{(p\chi_Q)^*(\la|Q|/2)}\le \int_{\la/2}^{\la}\a_Q^{(p\chi_Q)^*(t|Q|)}dt\le \int_0^{\la}\a_Q^{(p\chi_Q)^*(t|Q|)}dt.$$
It follows from this that (\ref{prom}) implies (\ref{eqbs}). Hence, $S_{\infty}'\Rightarrow S_{\infty}$, which, by Theorem \ref{cond}, proves that
$A_{p(\cdot)}\cap S'_{\infty}\Rightarrow {\mathcal P}$.
\end{proof}

Now, a simple comparison of two power functions allows us to get rid of sequences $\{\a_Q\}$ in the $S_{\infty}'$ condition.
We start with a preliminary proposition.

\begin{prop}\label{chareq} Let $K\ge 1$, $1<r\le p_-$, and $0<\la,\tau<1/2$. Given a cube $Q$, let $t_Q$ be the maximal root of the equation
\begin{equation}\label{impeq}
\tau^{1/r}(\la^{1/r}t)^{(p\chi_Q)^*((1-\tau)|Q|)}=K\la t^{(p\chi_Q)^*(\la|Q|)}.
\end{equation}
Then, for all $t\in (0,1)$,
\begin{eqnarray}
\tau^{1/r}(\la^{1/r}t)^{(p\chi_Q)^*((1-\tau)|Q|)}&\le& \tau^{1/r}(\la^{1/r}t_Q)^{(p\chi_Q)^*((1-\tau)|Q|)}\label{impineq}\\
&+&K\la t^{(p\chi_Q)^*(\la|Q|)}.\nonumber
\end{eqnarray}
\end{prop}

\begin{proof}
Since $K\ge 1$ and $r\le p_-$, we have that
$$\tau^{1/r}\la^{(p\chi_Q)^*((1-\tau)|Q|)/r}<K\la.$$
Also, since $0<\tau,\lambda<1/2$,
\begin{equation}\label{chare}
(p\chi_Q)^*((1-\tau)|Q|)\le (p\chi_Q)^*(\la|Q|).
\end{equation}
From this we have two cases. If both parts in (\ref{chare}) are equal, then (\ref{impeq}) has only one root $t_Q=0$, and then (\ref{impineq}) obviously holds.
If the inequality in (\ref{chare}) is strict, then (\ref{impeq}) has two roots, $t_0=0$ and $t_1$, where $0<t_1<1$. Denote the maximal root by $t_Q$.
Then for all $t\in [t_Q,1)$,
$$\tau^{1/r}(\la^{1/r}t)^{(p\chi_Q)^*((1-\tau)|Q|)}\le K\la t^{(p\chi_Q)^*(\la|Q|)},$$
from which (\ref{impineq}) follows.
\end{proof}

Denote by $t_Q(\la,\tau,r)$ the maximal root of (\ref{impeq}). It follows from (\ref{impeq}) that
\begin{equation}\label{tqeq}
t_Q(\la,\tau,r)=\left(\frac{\tau^{1/r}}{K}\la^{\frac{(p\chi_Q)^*((1-\tau)|Q|)}{r}-1}\right)^{\frac{1}{(p\chi_Q)^*(\la|Q|)-(p\chi_Q)^*((1-\tau)|Q|)}}.
\end{equation}
Consider also the following expression given by
\begin{equation}\label{smce}
\xi_Q(\la,\tau,r):=\left(\tau^{1/r}\la^{\frac{(p\chi_Q)^*((1-\tau)|Q|)-1}{r}}\right)^{\frac{1}{(p\chi_Q)^*(\la|Q|)-(p\chi_Q)^*((1-\tau)|Q|)}}.
\end{equation}

\begin{lemma}\label{aneq} Assume that $1<p_-\le p_+<\infty$. Then $p(\cdot)\in S_{\infty}'$ if and only if the following condition holds:
there exist $1<r\le p_-$, $0<\ga_0<\frac{1}{2}$ and $C\ge 1$
such that for all $0<\la,\tau<\ga_0$ and for any finite family of pairwise disjoint cubes ${\mathcal F}$,
\begin{equation}\label{sumfin}
\tau^{1/r}\sum_{Q\in {\mathcal F}}|Q|(\la^{1/r}\xi_Q(\la,\tau,r))^{(p\chi_Q)^*((1-\tau)|Q|)}\le C,
\end{equation}
where $\xi_Q(\la,\tau,r)$ is given by (\ref{smce}).
\end{lemma}

\begin{remark}\label{expl}
As we shall see, the result remains true with $t_Q(\la,\tau,r)$ instead of $\xi_Q(\la,\tau,r)$, and the proof is even more natural. The idea of introducing $\xi_Q(\la,\tau,r)$ is that after its substitution
into (\ref{sumfin}) we obtain a slightly more convenient expression compared to $t_Q(\la,\tau,r)$.
\end{remark}

\begin{proof}[Proof of Lemma \ref{aneq}]
We start with a simpler direction that the condition expressed in (\ref{sumfin}) implies the $S_{\infty}'$ condition.

Suppose that (\ref{sumfin}) holds. Let us show that a similar property holds also with $t_Q(\la,\tau,r)$ instead of $\xi_Q(\la,\tau,r)$.
Take $1<s<r$ such that
$$\frac{1}{s}\ge \frac{1}{r}+\frac{1}{p_-}\Big(1-\frac{1}{r}\Big).$$
It follows from this that
$$
\frac{(p\chi_Q)^*((1-\tau)|Q|)-1}{r}\le \frac{(p\chi_Q)^*((1-\tau)|Q|)}{s}-1.
$$
This, combined with the trivial estimate $\frac{\tau^{1/s}}{K}\le \tau^{1/r}$, implies
$$
\frac{\tau^{1/s}}{K}\la^{\frac{(p\chi_Q)^*((1-\tau)|Q|)}{s}-1}\le
\tau^{1/r}\la^{\frac{(p\chi_Q)^*((1-\tau)|Q|)-1}{r}},
$$
and hence
$$t_Q(\la,\tau,s)\le \xi_Q(\la,\tau,r).$$
Therefore, (\ref{sumfin}) implies that there exists $1<s<r$ such that for all $0<\la,\tau<\ga_0<1/2$,
$$
\tau^{1/s}\sum_{Q\in {\mathcal F}}|Q|(\la^{1/s}t_Q(\la,\tau,s))^{(p\chi_Q)^*((1-\tau)|Q|)}\le C.
$$

From this, given a sequence $\{\a_Q\}_{Q\in {\mathcal F}}$ such that $\a_Q\in (0,1)$ and
$$\la\sum_{Q\in {\mathcal F}}|Q|\a_Q^{(p\chi_Q)^*(\la|Q|)}\le 1,$$
using (\ref{impineq}), we obtain
\begin{eqnarray*}
&&\tau^{1/s}\sum_{Q\in {\mathcal F}}|Q|(\la^{1/s}\a_Q)^{(p\chi_Q)^*((1-\tau)|Q|)}\\
&&\le\tau^{1/s}\sum_{Q\in {\mathcal F}}|Q|(\la^{1/s}t_Q(\la,\tau,s))^{(p\chi_Q)^*((1-\tau)|Q|)}\\
&&+K\la\sum_{Q\in {\mathcal F}}|Q|\a_Q^{(p\chi_Q)^*(\la|Q|)}\le C+K.
\end{eqnarray*}
Thus, we have verified that the $S_{\infty}'$ condition holds.

Turn to the converse direction. The idea of the proof is taken from~\cite{D1}.
Assume that the $S_{\infty}'$ condition holds.
Define $K:=2^{\frac{p_+}{p_-}+1}C$, where $C$ is the constant from the $S_{\infty}'$ condition.
Observe that we may clearly assume that $C\ge 1$. Hence, $K>1$, and we can use the same reasoning about the equation (\ref{impeq}) as in the proof of Proposition \ref{chareq}.

Let us show that for any finite family of pairwise disjoint cubes ${\mathcal F}$,
\begin{equation}\label{sumfinn}
\tau^{1/r}\sum_{Q\in {\mathcal F}}|Q|(\la^{1/r}t_Q(\la,\tau,r))^{(p\chi_Q)^*((1-\tau)|Q|)}\le C,
\end{equation}
where $t_Q(\la,\tau,r)$ is given by (\ref{tqeq}).

For brevity denote $t_Q:=t_Q(\la,\tau,r)$. Set $t_Q':=\Big(\frac{1}{\la|Q|}\Big)^{\frac{1}{(p\chi_Q)^*(\la|Q|)}}$. We have $\la|Q|(t_Q')^{(p\chi_Q)^*(\la|Q|)}=1$, and hence, by the $S_{\infty}'$ condition,
$$\tau^{1/r}|Q|(\la^{1/r}t_Q')^{(p\chi_Q)^*((1-\tau)|Q|)}\le C,$$
which implies
$$\tau^{1/r}(\la^{1/r}t_Q')^{(p\chi_Q)^*((1-\tau)|Q|)}<K\la(t_Q')^{(p\chi_Q)^*(\la|Q|)}.$$
From this we conclude that $t_Q<t_Q'$.

Let ${\mathcal F}$ be an arbitrary finite family of pairwise disjoint cubes. Let ${\mathcal F}'\subset {\mathcal F}$ be the maximal subfamily for which
\begin{equation}\label{yanest}
\la\sum_{Q\in {\mathcal F}'}|Q| t_Q^{(p\chi_Q)^*(\la|Q|)}\le 2.
\end{equation}
We claim that ${\mathcal F}'={\mathcal F}$. Observe that from this, by the definition of $t_Q$, we would obtain (\ref{sumfinn}) with $C=2K$.

Suppose that this is not true, that is, ${\mathcal F}'\subset{\mathcal F}$. By (\ref{yanest}),
$$
\la\sum_{Q\in {\mathcal F}'}|Q|(t_Q/2^{1/p_-})^{(p\chi_Q)^*(\la|Q|)}\le 1.
$$
Hence, by the $S_{\infty}'$ condition,
$$
\tau^{1/r}\sum_{Q\in {\mathcal F}'}|Q|(\la^{1/r}t_Q/2^{1/p_-})^{(p\chi_Q)^*((1-\tau)|Q|)}\le C,
$$
and thus
$$
\tau^{1/r}\sum_{Q\in {\mathcal F}'}|Q|(\la^{1/r}t_Q)^{(p\chi_Q)^*((1-\tau)|Q|)}\le 2^{\frac{p_+}{p_-}}C.
$$
From this, using the definition of $t_Q$, we obtain that
\begin{equation}\label{half}
\la\sum_{Q\in {\mathcal F}'}|Q|t_Q^{(p\chi_Q)^*(\la|Q|)}\le \frac{1}{K}2^{\frac{p_+}{p_-}}C=\frac{1}{2}.
\end{equation}

Take now $t_R$, where $R\in {\mathcal F}\setminus {\mathcal F}'$. Since $t_R<t_R'$ we have
$$\la|R|(t_R)^{(p\chi_R)^*(\la|R|)}<1.$$
Therefore, by (\ref{half}),
$$
\la\sum_{Q\in {\mathcal F}'\cup\{R\}}|Q|t_Q^{(p\chi_Q)^*(\la|Q|)}\le \frac{3}{2},
$$
which contradicts the maximality of ${\mathcal F}'$ for which (\ref{yanest}) holds. This proves that ${\mathcal F}'={\mathcal F}$, and therefore (\ref{sumfinn}) is proved.

It remains to show that a similar relation to (\ref{sumfinn}) holds with $\xi_Q(\la,\tau,r)$ instead of $t_Q(\la,\tau,r)$.
Suppose that $1<s<r$. From this we have
$$\tau^{1/s}\le \frac{\tau^{1/r}}{K}\quad(0<\tau\le K^{-\frac{rs}{r-s}})$$
and
$$
\frac{(p\chi_Q)^*((1-\tau)|Q|)}{r}-1\le
\frac{(p\chi_Q)^*((1-\tau)|Q|)-1}{s},
$$
and therefore,
$$
\tau^{1/s}\la^{\frac{(p\chi_Q)^*((1-\tau)|Q|)-1}{s}}
\le \frac{\tau^{1/r}}{K}\la^{\frac{(p\chi_Q)^*((1-\tau)|Q|)}{r}-1}.
$$

Thus, for all $\tau\in (0, K^{-\frac{rs}{r-s}}]$ and $\la\in (0,1]$,
$$\xi_Q(\la,\tau,s)\le t_Q(\la,\tau,r),$$
which, along with (\ref{sumfinn}), proves (\ref{sumfin}) with $1<s<p_-$ instead of $r$ and for all $0<\la,\tau<\ga_0$, where $\ga_0:=K^{-\frac{rs}{r-s}}$.
\end{proof}

\subsection{Putting things together}
Let us now unify Theorem \ref{char2} and Lemma \ref{aneq}.

Given an exponent $p(\cdot)$ and a cube $Q$, for $\la,\tau\in (0,1)$ such that $\la+\tau<1$ define
$$
\Psi_{Q,p}(\la,\tau):=\frac{(p\chi_Q)^*((1-\tau)|Q|)}{(p\chi_Q)^*(\la|Q|)-(p\chi_Q)^*((1-\tau)|Q|)}.
$$

\begin{theorem}\label{reschar} Assume that $1<p_-\le p_+<\infty$. Then $p(\cdot)\in {\mathcal P}$ if and only if $p(\cdot)\in A_{p(\cdot)}$ and the following condition holds:
there exist $r>1$, $0<\ga_0<\frac{1}{2}$ and $C\ge 1$
such that for all $0<\la,\tau<\ga_0$ and for any finite family of pairwise disjoint cubes ${\mathcal F}$,
\begin{equation}\label{sumres}
\sum_{Q\in {\mathcal F}}|Q|\tau^{\frac{1}{r}(1+\Psi_{Q,p}(\la,\tau))}\la^{\frac{1}{r}(p(\chi_Q)^*(\la|Q|)-1)\Psi_{Q,p}(\la,\tau)}\le C.
\end{equation}
\end{theorem}

\begin{proof}
This is just a combination of Theorem \ref{char2} and Lemma \ref{aneq} along with a substitution of the expression for $\xi_Q(\la,\tau,r)$ into (\ref{sumfin}).
\end{proof}

The expression in (\ref{sumres}) does not look symmetric with respect to $\la$ and $\tau$. It turns out that it can be written more symmetric using the conjugate exponent $p'(\cdot)$.

\begin{theorem}\label{fstm}
Assume that $1<p_-\le p_+<\infty$. Then $p(\cdot)\in {\mathcal P}$ if and only if $p(\cdot)\in A_{p(\cdot)}$ and the following condition holds: there exist $r>1$, $0<\ga_0<1/2$ and $C\ge 1$ such that for
all $0<\la,\tau<\ga_0$ and for any sequence of pairwise disjoint cubes~${\mathcal F}$,
\begin{equation}\label{strf}
\sum_{Q\in {\mathcal F}}|Q|\tau^{\frac{1}{r}(1+\Psi_{Q,p}(\la,\tau))}\la^{\frac{1}{r}(1+\Psi_{Q,p'}(\tau,\la))}\le C.
\end{equation}
\end{theorem}

\begin{proof} By the second part of Lemma \ref{compos},
\begin{equation}\label{ptag}
(p'\chi_Q)^*(\la|Q|)=[(p\chi_Q)^*((1-\la)|Q|)]'
\end{equation}
for almost all $\la\in (0,1)$.
From this, a simple computation shows that for almost all $\la,\tau\in (0,1)$,
$$\big((p\chi_Q)^*(\la|Q|)-1\big)\Psi_{Q,p}(\la,\tau)=1+\Psi_{Q,p'}(\tau,\la),$$
which, by (\ref{sumres}), implies (\ref{strf}) for almost all $0<\la,\tau<\ga_0$.

Now, similarly to what we did in the proof of Theorem \ref{char2}, one can easily pass from ``almost all" to ``all" in (\ref{strf}). Indeed,
assume that $s\le \tau$ and $t\le \la$. Then
\begin{equation}\label{twosid}
\Psi_{Q,p}(t,s)\le \frac{(p\chi_Q)^*((1-\tau)|Q|)}{(p\chi_Q)^*(t|Q|)-(p\chi_Q)^*((1-\tau)|Q|)}\le \Psi_{Q,p}(\la,\tau).
\end{equation}

Suppose that $\la_0$ and $\tau_0$ do not belong to the set of almost all points for which (\ref{strf}) is already established. There exists a sequence $\e_k\to 0$ as $k\to \infty$ such that
$\la_0-\e_k$ and $\tau_0-\e_k$ do belong to this set for all $k$. Hence, using (\ref{twosid}), we obtain
$$
(\tau_0-\e_k)^{\Psi_{Q,p}(\la_0,\tau_0)}\le (\tau_0-\e_k)^{\Psi_{Q,p}(\la_0-\e_k,\tau_0-\e_k)}
$$
and
$$
(\la_0-\e_k)^{\Psi_{Q,p'}(\tau_0,\la_0))}\le (\la_0-\e_k)^{\Psi_{Q,p'}(\tau_0-\e_k,\la_0-\e_k))},
$$
from which
$$
\sum_{Q\in {\mathcal F}}|Q|(\tau_0-\e_k)^{\frac{1}{r}(1+\Psi_{Q,p}(\la_0,\tau_0))}(\la_0-\e_k)^{\frac{1}{r}(1+\Psi_{Q,p'}(\tau_0,\la_0))}\le C.
$$
Letting $k\to \infty$ yields (\ref{strf}) for $(\la,\tau)=(\la_0,\tau_0)$, and hence (\ref{strf}) holds for all $0<\la,\tau<\ga_0$.
\end{proof}

Next, we observe that yet another characterization holds, which is very similar to Theorem \ref{fstm}.

\begin{theorem}\label{fstm1}
Assume that $1<p_-\le p_+<\infty$. Then $p(\cdot)\in {\mathcal P}$ if and only if $p(\cdot)\in A_{p(\cdot)}$ and the following condition holds: there exist $r>1$, $0<\ga_0<1/2$ and $C\ge 1$ such that for
all $0<\la,\tau<\ga_0$ and for any sequence of pairwise disjoint cubes~${\mathcal F}$,
\begin{equation}\label{weakf}
\sum_{Q\in {\mathcal F}}|Q|(\tau\lambda)\tau^{\frac{1}{r}\Psi_{Q,p}(\la,\tau)}\la^{\frac{1}{r}\Psi_{Q,p'}(\tau,\la)}\le C.
\end{equation}
\end{theorem}

\begin{proof}
We trivially have that (\ref{strf}) implies (\ref{weakf}). We will show that, in turn, (\ref{weakf}) with a given $r>1$ implies (\ref{strf}) with some $1<s<r$ instead of $r$.

Define $\nu$ by
$$\frac{1}{\nu}:=\Big(\frac{p_+-p_-}{p_+}\Big)\Big(1-\frac{1}{r}\Big)+\frac{1}{r}.$$
Observe that $1<\nu<r$. It follows from this definition that
\begin{equation}\label{phiq}
\frac{1}{\nu}(1+\Psi_{Q,p}(\la,\tau))\ge 1+\frac{1}{r}\Psi_{Q,p}(\la,\tau).
\end{equation}
Indeed, (\ref{phiq}) is equivalent to that
$$
\Big(\frac{1}{\nu}-\frac{1}{r}\Big)\Psi_{Q,p}(\la,\tau)\ge 1-\frac{1}{\nu}.
$$
But this follows from the estimate
$$
\Psi_{Q,p}(\la,\tau)\ge \frac{p_-}{(p\chi_Q)^*(\la|Q|)-p_-}\ge \frac{p_-}{p_+-p_-}
$$
and the fact that
$$\Big(\frac{1}{\nu}-\frac{1}{r}\Big)\frac{p_-}{p_+-p_-}=1-\frac{1}{\nu}.$$

Similarly, setting
$$\frac{1}{\theta}:=\Big(\frac{(p')_+-(p')_-}{(p')_+}\Big)\Big(1-\frac{1}{r}\Big)+\frac{1}{r},$$
we obtain that $1<\theta<r$ and
$$
\frac{1}{\theta}(1+\Psi_{Q,p'}(\tau,\la))\ge 1+\frac{1}{r}\Psi_{Q,p'}(\tau,\la).
$$
From this and (\ref{phiq}) we obtain that (\ref{strf}) holds with $s:=\min(\nu,\theta)$ instead of $r$.
\end{proof}

The whole difference between (\ref{strf}) and (\ref{weakf}) is that $(\tau\la)^{1/r}$ in (\ref{strf}) is replaced by a smaller expression $\tau\la$ in (\ref{weakf}).
Thus, (\ref{weakf}) is easier to check. This will be useful in the proof of the following statement, which is yet another characterization of the class ${\mathcal P}$.

\begin{theorem}\label{yac}
Assume that $1<p_-\le p_+<\infty$. Then $p(\cdot)\in {\mathcal P}$ if and only if $p(\cdot)\in A_{p(\cdot)}$ and the following condition holds: there exist $r>1$, $0<\ga_0<1/2$ and $C\ge 1$ such that for
all $0<\la,\tau<\ga_0$ and for any sequence of pairwise disjoint cubes~${\mathcal F}$,
\begin{equation}\label{intcon}
\sum_{Q\in {\mathcal F}}|Q|\int_{\la}^{\ga_0}\int_{\tau}^{\ga_0}
\tau^{\frac{1}{r}\Psi_{Q,p}(t,s)}\la^{\frac{1}{r}\Psi_{Q,p'}(s,t)}ds dt\le C.
\end{equation}
\end{theorem}

\begin{proof}
This is a simple combination of Theorems \ref{fstm} and \ref{fstm1}. Indeed, it suffices to show that $(\ref{strf})\Rightarrow (\ref{intcon})\Rightarrow (\ref{weakf})$.
Assume that (\ref{strf}) holds. Then, for all $0<t,s<\ga_0$,
$$
\sum_{Q\in {\mathcal F}}|Q|s^{\frac{1}{r}\Psi_{Q,p}(t,s)}t^{\frac{1}{r}\Psi_{Q,p'}(s,t)}\le \frac{C}{(ts)^{1/r}}.
$$
Integrating both sides over $t,s\in (0,\ga_0)$ yields
$$
\sum_{Q\in {\mathcal F}}|Q|\int_0^{\ga_0}\int_0^{\ga_0}s^{\frac{1}{r}\Psi_{Q,p}(t,s)}t^{\frac{1}{r}\Psi_{Q,p'}(s,t)}ds dt\le C.
$$
From this, since
$$
\int_{\la}^{\ga_0}\int_{\tau}^{\ga_0}
\tau^{\frac{1}{r}\Psi_{Q,p}(t,s)}\la^{\frac{1}{r}\Psi_{Q,p'}(s,t)}ds dt\le
\int_0^{\ga_0}\int_0^{\ga_0}s^{\frac{1}{r}\Psi_{Q,p}(t,s)}\eta^{\frac{1}{r}\Psi_{Q,p'}(s,t)}ds dt,
$$
we obtain (\ref{intcon}).

Suppose now that (\ref{intcon}) holds. Then, by (\ref{twosid}), for all $0<\la,\tau<\ga_0$,
\begin{eqnarray*}
&&\sum_{Q\in {\mathcal F}}|Q|\frac{\la\tau}{4}(\tau/2)^{\frac{1}{r}\Psi_{Q,p}(\la,\tau)}(\la/2)^{\frac{1}{r}\Psi_{Q,p'}(\tau,\la)}\\
&&\le\sum_{Q\in {\mathcal F}}|Q|\int_{\la/2}^{\la}\int_{\tau/2}^{\tau}
(\tau/2)^{\frac{1}{r}\Psi_{Q,p}(t,s)}(\la/2)^{\frac{1}{r}\Psi_{Q,p'}(s,t)}ds dt\le C,
\end{eqnarray*}
from which (\ref{weakf}) follows.
\end{proof}

\section{Proof of Theorem \ref{mr}}\label{ss5}
Let $F_{p,\la,\tau}$ be the function from Definition \ref{fpla}. Fix a cube $Q\subset {\mathbb R}^n$, and consider the iterated rearrangement
of $F_{p,\la,\tau}$ on $Q\times Q$ as it was defined in Section \ref{ss32}.

\begin{lemma}\label{computr} For almost all $t,s\in (0,1)$,
\begin{equation}\label{mid}
(F_{p,\la,\tau}\chi_{Q\times Q})^*(t|Q|,s|Q|)=
\begin{cases}
\tau^{\Psi_{Q,p}(t,s)}\la^{\Psi_{Q,p'}(s,t)},&t+s<1\\
0,&t+s\ge 1.
\end{cases}
\end{equation}
\end{lemma}

\begin{proof} For brevity denote $F$ instead of $F_{p,\la,\tau}$, and $F^*(t|Q|,s|Q|)$ instead of the left-hand side of (\ref{mid}).

For $\a,\b\in (1,\infty)$ consider the function
$$\psi(\a,\b):=
\begin{cases}
\tau^{\frac{\b}{\a-\b}}\la^{\frac{\a'}{\b'-\a'}},&\a>\b\\
0,&\a\le \b.
\end{cases}
$$
It is easily seen that $\psi$ is strictly increasing in $\a$ on $[\b,\infty)$ and strictly decreasing in $\b$ on $(1,\a]$.

For a fixed $y\in Q$, denote
$$\f(\a):=\psi\big(\a,p(y)\big).$$
Then
$$F(x,y)=\f\big(p(x)\big)\chi_{\{x\in Q:p(x)>p(y)\}}(x),$$
and $\f$ is strongly increasing on $[p(y),\infty)$.
Therefore, by the first part of Lemma \ref{comp1}, for all $t\in (0,1)$,
$$F^*(t|Q|,y)=\psi\big((p\chi_Q)^*(t|Q|),p(y)\big)\chi_{\{t\in (0,1):(p\chi_Q)^*(t|Q|)>p(y)\}}(t).$$

Now, for a fixed $t\in (0,1)$, denote
$$\eta(\b):=\psi\big((p\chi_Q)^*(t|Q|),\b\big).$$
Then
$$F^*(t|Q|,y)=\eta(p(y))\chi_{\{y\in Q: p(y)<(p\chi_Q)^*(t|Q|)\}}(y),$$
and $\eta$ is strongly decreasing on $(1, (p\chi_Q)^*(t|Q|)]$.
Therefore, by the second part of Lemma \ref{comp1}, for almost all $s\in (0,1)$,
\begin{equation}\label{fts}
F^*(t|Q|,s|Q|)=\psi\big((p\chi_Q)^*(t|Q|),(p\chi_Q)^*((1-s)|Q|)\big).
\end{equation}

Suppose now that $t+s\ge 1$. Then
$$(p\chi_Q)^*(t|Q|)\le (p\chi_Q)^*((1-s)|Q|),$$
which, by (\ref{fts}), implies that $F^*(t|Q|,s|Q|)=0$, and hence (\ref{mid}) holds.

Suppose that $t+s<1$. Then
$$(p\chi_Q)^*(t|Q|)\ge (p\chi_Q)^*((1-s)|Q|).$$
If we have an equality here, (\ref{fts}) shows again that $F^*(t|Q|,s|Q|)=0$, which confirms (\ref{mid}), since in this case $\tau^{\Psi_{Q,p}(t,s)}=0$.
If the inequality here is strict, then, applying (\ref{ptag}), we obtain
\begin{equation}\label{psibig}
\psi\big((p\chi_Q)^*(t|Q|),(p\chi_Q)^*((1-s)|Q|)\big)=\tau^{\Psi_{Q,p}(t,s)}\la^{\Psi_{Q,p'}(s,t)},
\end{equation}
which, along with (\ref{fts}), proves (\ref{mid}).
\end{proof}

We are now ready to prove Theorem \ref{mr}.

\begin{proof}[Proof of Theorem \ref{mr}]
The proof goes via Theorem \ref{yac} by showing that the ${\mathcal U}_{\infty}$ condition is equivalent to the condition expressed in (\ref{intcon}).

Suppose that the ${\mathcal U}_{\infty}$ condition holds. Recall that this means that
for some $0<\ga<1/2$ and for all $0<\la,\tau<\ga$,
$$\sum_{Q\in {\mathcal F}}\frac{1}{|Q|}\inf_{E\subset Q:|E|=\la|Q|\atop G\subset Q:|G|=\tau|Q|}\int_{Q\setminus G}\int_{Q\setminus E}
F_{p,\la,\tau}(x,y)^{1/r}dxdy\le C.$$
From this, by Proposition \ref{itrpr},
$$
\sum_{Q\in {\mathcal F}}|Q|\int_{\la}^{1}\int_{\tau}^{1}(F_{p,\la,\tau}\chi_{Q\times Q})^*(t|Q|,s|Q|)dsdt\le C.
$$
Hence, by Lemma \ref{computr}, for $\ga_0:=\ga$ and for all $0<\la,\tau<\ga_0$ we have
$$
\sum_{Q\in {\mathcal F}}|Q|\int_{\la}^{\ga_0}\int_{\tau}^{\ga_0}\tau^{\frac{1}{r}\Psi_{Q,p}(t,s)}\la^{\frac{1}{r}\Psi_{Q,p'}(s,t)}dsdt\le C.
$$
Thus, the condition (\ref{intcon}) holds.

Let us turn to the converse direction. Assume that (\ref{intcon}) holds.
Set $\ga:=\ga_0/2$, and assume that $0<\la,\tau<\ga$.

Let ${\mathcal F}$ be a finite family of pairwise disjoint cubes, and let $Q\in {\mathcal F}$.
Choose subsets
$$E_{\la}\subset \{x\in Q:p(x)\ge (p\chi_Q)^*(\la|Q|)\}$$
and
$$G_{\tau}\subset \{y\in Q: p(y)\le (p\chi_Q)^*((1-\tau)|Q|)\}$$
such that $|E_{\la}|=\la|Q|$ and $|G_{\tau}|=\tau|Q|$. In order to prove that the ${\mathcal U}_{\infty}$ condition holds, it suffices to show that
$$I_{p,\la,\tau}:=\int_{Q\setminus G_{\tau}}\int_{Q\setminus E_{\la}}
F_{p,\la,\tau}(x,y)^{1/r}dxdy$$
satisfies
\begin{equation}\label{itss}
I_{p,\la,\tau}\le C|Q|^2\int_{\la}^{\ga_0}\int_{\tau}^{\ga_0}\tau^{\frac{1}{r}\Psi_{Q,p}(t,s)}\la^{\frac{1}{r}\Psi_{Q,p'}(s,t)}dsdt,
\end{equation}
where $C$ does not depend on $\la$ and $\tau$.

Write $Q\setminus E_{\la}$ as a disjoint union $Q\setminus E_{\la}=A\cup B$, where
$$A\subset \{x\in Q: (p\chi_Q)^*(\ga_0|Q|)\le p(x)\le (p\chi_Q)^*(\la|Q|)\}$$
and
$$B\subset \{x\in Q: p(x)\le (p\chi_Q)^*(\ga_0|Q|)\}$$
with $|A|=(\ga_0-\la)|Q|$ and $|B|=(1-\ga_0)|Q|$.

Similarly, write $Q\setminus G_{\tau}$ as a disjoint union $Q\setminus G_{\tau}=C\cup D$, where
$$C\subset \{y\in Q: (p\chi_Q)^*((1-\tau)|Q|)\le p(y)\le (p\chi_Q)^*((1-\ga_0)|Q|)\}$$
and
$$D\subset \{y\in Q: p(y)\ge (p\chi_Q)^*((1-\ga_0)|Q|)\}$$
with $|C|=(\ga_0-\tau)|Q|$ and $|D|=(1-\ga_0)|Q|$.

Split $I_{p,\la,\tau}$
into four parts
\begin{eqnarray*}
&&I_{p,\la,\tau}=\int_{C}\int_{A}
F_{p,\la,\tau}(x,y)^{1/r}dxdy+\int_{C}\int_{B}
F_{p,\la,\tau}(x,y)^{1/r}dxdy\\
&&+\int_{D}\int_{A}
F_{p,\la,\tau}(x,y)^{1/r}dxdy+\int_{D}\int_{B}
F_{p,\la,\tau}(x,y)^{1/r}dxdy.
\end{eqnarray*}

Let $\psi$ be the function defined in the proof of Lemma \ref{computr}. Using the monotonicity of $\psi$, we obtain
\begin{eqnarray*}
\int_{C}\int_{B}
F_{p,\la,\tau}(x,y)^{1/r}dxdy&=&\int_C\int_{\{x\in B:p(x)\ge p(y)\}}\psi\big(p(x),p(y)\big)^{1/r}dxdy\\
&\le& |B|\int_{C}\psi\big((p\chi_Q)^*(\la|Q|),p(y)\big)^{1/r}dy
\end{eqnarray*}
and
\begin{eqnarray*}
\int_{C}\int_{A}
F_{p,\la,\tau}(x,y)^{1/r}dxdy&=&\int_{C}\int_{A}\psi\big(p(x),p(y)\big)^{1/r}dxdy\\
&\ge& |A|\int_{C}\psi\big((p\chi_Q)^*(\la|Q|),p(y)\big)^{1/r}dy.
\end{eqnarray*}
Since $|A|=(\ga_0-\la)|Q|\ge \ga_0|Q|/2$, we conclude that
\begin{eqnarray}
&&\int_{C}\int_{B}
F_{p,\la,\tau}(x,y)^{1/r}dxdy\label{CB}\\
&&\le \frac{2(1-\ga_0)}{\ga_0}\int_{C}\int_{A}
F_{p,\la,\tau}(x,y)^{1/r}dxdy. \nonumber
\end{eqnarray}

Similarly,
$$\int_{D}\int_{A}
F_{p,\la,\tau}(x,y)^{1/r}dxdy\le |D|\int_A\psi\big(p(x),(p\chi_{Q})^*((1-\ga_0)|Q|)\big)^{1/r}dx$$
and
$$\int_{C}\int_{A}
F_{p,\la,\tau}(x,y)^{1/r}dxdy\ge |C|\int_A\psi\big(p(x),(p\chi_{Q})^*((1-\ga_0)|Q|)\big)^{1/r}dx,$$
from which
\begin{eqnarray}
&&\int_{D}\int_{A}
F_{p,\la,\tau}(x,y)^{1/r}dxdy\label{DA}\\
&&\le \frac{2(1-\ga_0)}{\ga_0}\int_{C}\int_{A}
F_{p,\la,\tau}(x,y)^{1/r}dxdy. \nonumber
\end{eqnarray}

Further,
\begin{eqnarray*}
&&\int_{D}\int_{B}
F_{p,\la,\tau}(x,y)^{1/r}dxdy\\
&&\le |D||B|\psi\big((p\chi_Q)^*(\ga_0|Q|),(p\chi_Q)^*((1-\ga_0)|Q|)\big)^{1/r}
\end{eqnarray*}
and
\begin{eqnarray*}
&&\int_{C}\int_{A}
F_{p,\la,\tau}(x,y)^{1/r}dxdy\\
&&\ge |C||A|\psi\big((p\chi_Q)^*(\ga_0|Q|),(p\chi_Q)^*((1-\ga_0)|Q|)\big)^{1/r},
\end{eqnarray*}
from which
\begin{eqnarray}
&&\int_{D}\int_{B}
F_{p,\la,\tau}(x,y)^{1/r}dxdy\label{DB}\\
&&\le \frac{4(1-\ga_0)^2}{\ga_0^2}\int_{C}\int_{A}
F_{p,\la,\tau}(x,y)^{1/r}dxdy. \nonumber
\end{eqnarray}
Combining (\ref{CB}), (\ref{DA}) and (\ref{DB}) yields
\begin{eqnarray*}
I_{p,\la,\tau}&\le& C_{\ga_0}\int_{C}\int_{A}
F_{p,\la,\tau}(x,y)^{1/r}dxdy\\
&=&C_{\ga_0}\int_C\int_A\psi\big(p(x),p(y)\big)^{1/r}dxdy.
\end{eqnarray*}

Fix an $y\in C$, and consider
$$I(y):=\int_A\psi\big(p(x),p(y)\big)^{1/r}dx.$$
Observe that since $(p\chi_Q)^*(\ga_0|Q|)\ge (p\chi_Q)^*((1-\ga_0)|Q|)$, we have $p(x)\ge p(y)$ for all $x\in A$.

Consider the case where $(p\chi_Q)^*(\ga_0|Q|)>p(y)$. Then $p(x)>p(y)$ for all $x\in A$.
Since the function
$$\f(\a):=\psi\big(\a,p(y)\big)$$
is strictly increasing for $\a\ge p(y)$,
by the first part of Lemma \ref{compos}, the sets $Q\setminus B$ and $E_{\la}$ can be represented as subsets
$$Q\setminus B\subset\{x\in Q:\f(p(x))\ge (\f(p)\chi_Q)^*(\ga_0|Q|)\}$$
and
$$E_{\la}\subset\{x\in Q:\f(p(x))\ge (\f(p)\chi_Q)^*(\la|Q|)\}.$$
Since $|Q\setminus B|=\ga_0|Q|$ and $|E_{\la}|=\la|Q|$, using Remark \ref{bsh} and applying again Lemma \ref{compos}, we obtain
\begin{eqnarray*}
I(y)&=&\int_{Q\setminus B}\f(p(x))^{1/r}dx-\int_{E_{\la}}\f(p(x))^{1/r}dx\\
&=&\int_{0}^{\ga_0|Q|}(\f(p)\chi_Q)^*(t)^{1/r}dt-\int_{0}^{\la|Q|}(\f(p)\chi_Q)^*(t)^{1/r}dt\\
&=&\int_{\la|Q|}^{\ga_0|Q|}(\f(p)\chi_Q)^*(t)^{1/r}dt=|Q|\int_{\la}^{\ga_0}\psi\big((p\chi_Q)^*(t|Q|),p(y)\big)^{1/r}dt.
\end{eqnarray*}

Suppose now that $(p\chi_Q)^*(\ga_0|Q|)=p(y)$. Here we have two subcases. If $(p\chi_Q)^*(t|Q|)=p(y)$ for all $t\in [\la,\ga_0]$, then $p(x)=p(y)$ for all $x\in A$, and hence $I(y)=0$.
Otherwise, define
$$t_0:=\sup\{t\in [\la,\ga_0]:(p\chi_Q)^*(t|Q|)>p(y)\}.$$
Since the rearrangement is right-continuous, we obtain that $t_0\in (\la,\ga_0)$, and, moreover $(p\chi_Q)^*(t_0|Q|)=p(y)$ and $(p\chi_Q)^*(t|Q|)>p(y)$ for all $t\in [\la,t_0)$.
Thus, by absolute continuity, we have
$$I(y)=\lim_{\e\to 0}\int_{A_{\e}}\psi\big(p(x),p(y)\big)^{1/r}dx,$$
where
$$A_{\e}\subset \{x\in Q: (p\chi_Q)^*((t_0-\e)|Q|)\le p(x)\le (p\chi_Q)^*(\la|Q|)\}$$
with $|A_{\e}|=(t_0-\e-\la)|Q|$.
From this, exactly as above, we obtain
\begin{eqnarray*}
I(y)&=&\lim_{\e\to 0}|Q|\int_{\la}^{t_0-\e}\psi\big((p\chi_Q)^*(t|Q|),p(y)\big)^{1/r}dt\\
&=&|Q|\int_{\la}^{t_0-\e}\psi\big((p\chi_Q)^*(t|Q|),p(y)\big)^{1/r}dt.
\end{eqnarray*}

Unifying all the cases considered above yields
$$
I(y)\le |Q|\int_{\la}^{\ga_0}\psi\big((p\chi_Q)^*(t|Q|),p(y)\big)^{1/r}dt \quad(y\in C).
$$
Hence,
\begin{equation}\label{iplt}
I_{p,\la,\tau}\le C_{\ga_0}|Q|\int_{\la}^{\ga_0}\int_C\psi\big((p\chi_Q)^*(t|Q|),p(y)\big)^{1/r}dydt.
\end{equation}

Fix now a $t\in [\la,\ga_0]$ and consider
$$J(t):=\int_C\psi\big((p\chi_Q)^*(t|Q|),p(y)\big)^{1/r}dy.$$
To estimate this part, we use almost the same ideas as for $I(y)$ except minor technical details which we outline briefly.

Consider the case where $(p\chi_Q)^*((1-\ga_0)|Q|)<(p\chi_Q)^*(t|Q|)$. Define
$$\eta(\b):=\psi\big((p\chi_Q)^*(t|Q|),\b\big).$$
Then the sets $Q\setminus D$ and $G_{\tau}$ can be represented as subsets
$$Q\setminus D\subset \{y\in Q: \eta(p(y))\ge \eta\big((p\chi_Q)^*((1-\ga_0)|Q|)\big)\}$$
and
$$G_{\tau}\subset \{y\in Q: \eta(p(y))\ge \eta\big((p\chi_Q)^*((1-\tau)|Q|)\big)\}.$$

Now, the minor technical difference with the estimate for $I(y)$ is that, by the second part of Lemma \ref{compos},
the equality
\begin{equation}\label{imeq}
\eta\big((p\chi_Q)^*((1-s)|Q|)\big)=(\eta(p)\chi_Q)^*(s|Q|)
\end{equation}
holds for almost all $s\in (0,1)$. If (\ref{imeq}) holds for $s=\ga_0$ and $s=\tau$, then, exactly as above we obtain that
\begin{eqnarray}
J(t)&=&\int_{Q\setminus D}\eta(p(y))^{1/r}dy-\int_{G_{\tau}}\eta(p(y))^{1/r}dy\nonumber\\
&=&\int_{0}^{\ga_0|Q|}(\eta(p)\chi_Q)^*(s)^{1/r}ds-\int_{0}^{\la|Q|}(\eta(p)\chi_Q)^*(s)^{1/r}ds\nonumber\\
&=&\int_{\tau|Q|}^{\ga_0|Q|}\eta\big((p\chi_Q)^*(|Q|-s)\big)^{1/r}ds\nonumber\\
&=&|Q|\int_{\tau}^{\ga_0}\psi\big((p\chi_Q)^*(t|Q|),(p\chi_Q)^*((1-s)|Q|)\big)^{1/r}ds.\label{jt}
\end{eqnarray}
In the case if (\ref{imeq}) does not hold for $s=\ga_0$ or for $s=\tau$ (or for both of them), we can find sequences $\{\ga_k\}$ and $\{\tau_k\}$
such that $\ga_k\to \ga_0$ and $\tau_k\to \tau$ as $k\to \infty$ and (\ref{imeq}) holds for $s=\ga_k$ and $s=\eta_k$ for all $k$. Then, changing slightly the sets
$Q\setminus D$ and $G_{\tau}$ and using absolute continuity of the corresponding integrals, we will obtain the same estimate for $J(t)$.

Suppose now that $(p\chi_Q)^*((1-\ga_0)|Q|)=(p\chi_Q)^*(t|Q|)$. Then, exactly as above, we obtain that if
$(p\chi_Q)^*((1-s)|Q|)=(p\chi_Q)^*(t|Q|)$ for all $s\in [\tau,\ga_0]$, then $J(t)=0$, and, otherwise (\ref{jt}) holds for $t_0$ instead of $\ga_0$,
where
$$
t_0:=\sup\{s\in [\tau,\ga_0]: (p\chi_Q)^*((1-s)|Q|)<(p\chi_Q)^*(t|Q|)\}.
$$
In all cases we obtain that $J(t)$ is bounded by the expression in (\ref{jt}).

Combining (\ref{jt}) with (\ref{iplt}) yields
$$
I_{p,\la,\tau}\le C_{\ga_0}|Q|^2\int_{\la}^{\ga_0}\int_{\tau}^{\ga_0}\psi\big((p\chi_Q)^*(t|Q|),(p\chi_Q)^*((1-s)|Q|)\big)^{1/r}dsdt,
$$
which, along with (\ref{psibig}), proves (\ref{itss}), and therefore, the proof is complete.
\end{proof}

\begin{remark}\label{repap}
As we mentioned in the Introduction, the $A_{p(\cdot)}$ condition is a natural replacement of the $LH_0$ condition in earlier sufficient conditions for $p(\cdot)\in {\mathcal P}$ as it is trivially necessary.

On the other hand, from the practical point of view, the $LH_0$ condition is much easier to check. Therefore, it is natural to ask whether the $A_{p(\cdot)}$ condition in Theorem \ref{mr} can be changed by the
$LH_0$ condition in order to provide an easier to check sufficient condition (but which will not be necessary). Observe that this cannot be obtained directly from Theorem \ref{mr}, since the $LH_0$ condition
does not imply the $A_{p(\cdot)}$ condition, in general (see, e.g., \cite[Ex. 4.58]{CUF}).

Nevertheless, the answer to this question is positive. Indeed, the $A_{p(\cdot)}$ condition was used only once in the proof of Theorem \ref{char1} in order to establish (\ref{wws}). To be more precise, (\ref{wws}) represents
the following property: there exists $C>0$ such that if $\|f\|_{L^{p(\cdot)}}\le 1$ and $f_1:=f\chi_{\{f\ge 1\}}$, then $\|Mf_1\|_{L^{p(\cdot)}}\le C$. However, it is well known that the $LH_0$ condition is sufficient for
this property (see, e.g., in \cite[p. 99]{CUF}).

Hence, the proof of Theorem \ref{char1} shows that $LH_0\cap S_{\infty}\Rightarrow {\mathcal P}$. It is straightforward to check that the $A_{p(\cdot)}$ condition can be replaced by the $LH_0$ condition in all subsequent statements, providing only the sufficiency part. Thus, we have the following.

\begin{prop}\label{stlh}
Assume that $1<p_-\le p_+<\infty$. If $p(\cdot)\in LH_0\cap {\mathcal U}_{\infty}$, then $p(\cdot)\in {\mathcal P}$.
\end{prop}

We emphasize that this Proposition can be more useful that Theorem~\ref{mr} in some situations but of course Theorem \ref{mr} is more general.
\end{remark}

\section{More about the ${\mathcal U}_{\infty}$ condition}\label{ss6}
In this section we consider different aspects related to the ${\mathcal U}_{\infty}$ condition. In particular, we obtain an easy to check but still powerful condition
implying that $p(\cdot)\in {\mathcal U}_{\infty}$.
We also provide several new examples of the exponents $p(\cdot)\in {\mathcal P}$.

\subsection{A comparison with the $N_{\infty}$ condition}\label{ss61}
In order to feel better the  ${\mathcal U}_{\infty}$ condition, let us show that this condition follows trivially from
the ${N}_{\infty}$ condition defined in Definition \ref{nekc}. Indeed, the proof below shows that in the case where $p(\cdot)\in N_{\infty}$,
we even do not need to subtract the sets $G\subset Q$ and $E\subset Q$ in the definition of the ${\mathcal U}_{\infty}$ condition, namely,
$$
\inf_{E\subset Q:|E|=\la|Q|\atop G\subset Q:|G|=\tau|Q|}\int_{Q\setminus G}\int_{Q\setminus E}F_{p,\la,\tau}(x,y)^{1/r}dxdy
$$
can be replaced by
$$
\int_{Q}\int_{Q}F_{p,\la,\tau}(x,y)^{1/r}dxdy.
$$

Let us start with the following auxiliary statement.

\begin{prop}\label{convl} Assume that $0<\d\le {\rm e}^{-8p_+}$. Then, for any cube $Q\subset {\mathbb R}^n$,
\begin{equation}\label{doubc}
\int_Q\int_Q\d^{\frac{1}{|p(x)-p(y)|}}dxdy\le |Q|\inf_{\nu\in [p_-,p_+]}\int_Q\d^{\frac{1}{2|p(x)-\nu|}}dx.
\end{equation}
\end{prop}

\begin{proof}
For $\d\in (0,1)$ and $t>0$ consider the function $\f(t):=\d^{1/t}$. Observing that
$$\f''(t)=\Big(\log\frac{1}{\d}\Big)t^{-3}\d^{1/t}\Big(\frac{\log(1/\d)}{x}-2\Big),$$
we have that $\f$ is convex on $\Big(0, \frac{\log(1/\d)}{2}\Big]$.
From this, assuming that $\nu\in [p_-,p_+]$ and $8p_+\le \log(1/\d)$, we obtain
\begin{eqnarray*}
\d^{\frac{1}{|p(x)-p(y)|}}&\le& \d^{\frac{1}{|p(x)-\nu|+|p(y)-\nu|}}\\
&\le& \frac{1}{2}\Big(\d^{\frac{1}{2|p(x)-\nu|}}+\d^{\frac{1}{2|p(y)-\nu|}}\Big),
\end{eqnarray*}
from which (\ref{doubc}) follows.
\end{proof}

\begin{prop}\label{nekimp} Assume that $1<p_-\le p_+<\infty$. Then $p(\cdot)\in N_{\infty}\Rightarrow p(\cdot)\in {\mathcal U}_{\infty}$.
\end{prop}

\begin{proof} By the $N_{\infty}$ condition, there exist $0<c<1$ and $p_{\infty}>0$ such that
\begin{equation}\label{nc}
\int_{{\mathbb R}^n}c^{\frac{1}{|p(x)-p_{\infty}|}}dx<\infty.
\end{equation}
We make an obvious observation that if (\ref{nc}) holds, then $p_{\infty}\in [p_-,p_+]$.

Let $0<\ga<1/2$ to be chosen and let $0<\la,\tau<\ga$. Then
\begin{eqnarray*}
F_{p,\la,\tau}(x,y)&=&\big(\tau^{p(y)}\la^{p(x)(p(y)-1)}\big)^{\frac{1}{p(x)-p(y)}}\chi_{\{(x,y):p(x)>p(y)\}}\\
&\le& \ga^{\frac{{p_-^2}}{|p(x)-p(y)|}}.
\end{eqnarray*}
Let $r>1$. Take $\ga$ such that
$$\ga^{p_-^2/r}<\min({\rm e}^{-8p_+},c^2),$$
where $c$ is from (\ref{nc}). Then, by Proposition \ref{convl},
\begin{eqnarray*}
\int_Q\int_QF_{p,\la,\tau}(x,y)^{1/r}dxdy&\le& |Q|\inf_{\nu\in [p_-,p_+]}\int_Q\ga^{\frac{p_-^2}{2r|p(x)-\nu|}}dx\\
&\le& |Q|\inf_{\nu\in [p_-,p_+]}\int_Qc^{\frac{1}{|p(x)-\nu|}}dx.
\end{eqnarray*}
Therefore, for every finite family of pairwise disjoint cubes ${\mathcal F}$,
\begin{eqnarray*}
&&\sum_{Q\in {\mathcal F}}\frac{1}{|Q|}\inf_{E\subset Q:|E|=\la|Q|\atop G\subset Q:|G|=\tau|Q|}\int_{Q\setminus G}\int_{Q\setminus E}F_{p,\la,\tau}(x,y)^{1/r}dxdy\\
&&\le \sum_{Q\in {\mathcal F}}\frac{1}{|Q|}\int_{Q}\int_{Q}F_{p,\la,\tau}(x,y)^{1/r}dxdy\\
&&\le \sum_{Q\in {\mathcal F}}\inf_{\nu\in [p_-,p_+]}\int_Qc^{\frac{1}{|p(x)-\nu|}}dx\le \int_{{\mathbb R}^n}c^{\frac{1}{|p(x)-p_{\infty}|}}dx\le C,
\end{eqnarray*}
and thus, the ${\mathcal U}_{\infty}$ condition holds.
\end{proof}

\subsection{A useful sufficient condition}\label{ss62}
In this section we establish a much more powerful (compared to $N_{\infty}$) condition implying $p(\cdot)\in {\mathcal U}_{\infty}$, where it is crucial that we can remove subsets $E,G\subset Q$
in the definition of ${\mathcal U}_{\infty}$.
We start with the following auxiliary geometric statement.

\begin{lemma}\label{cubes}
Let $Q\subset {\mathbb R}^n$ be a cube such that $0\not\in Q$, and let $0<\d<1$. Then there exists a subset $E\subset Q$ with $|E|=\d|Q|$ such that
$$\frac{|x|}{|y|}\le \frac{\sqrt{n}}{\d}$$
for all $x\in Q$ and $y\in Q\setminus E$.
\end{lemma}

\begin{proof} Denote $Q:=\prod_{j=1}^n(a_j,a_j+h), h>0$.
Take $\xi\in \partial Q$ such that $\inf_{z\in Q}|z|=|\xi|.$

Consider the case where $Q$ is contained in one of $2^n$ quadrants.
Then for every $1\le j\le n$ we have that either $a_j\ge 0$ or $a_j+h\le 0$.
Also, by a simple geometry, $\xi$ is one of the vertices of $\overline Q$, and for each coordinate $\xi_j$ of $\xi$ we have
that
$$
\xi_j=
\begin{cases}
a_j,& a_j\ge 0\\
a_j+h,& a_j+h\le 0
\end{cases}
$$
for every $j=1,\dots,n$.

Suppose that $\max_{1\le j\le n}|\xi_j|=|\xi_{j_0}|$. If $a_{j_0}\ge 0$, take
$$E:=\prod_{j\not=j_0}(a_j,a_j+h)\times (a_{j_0},a_{j_0}+\d h).$$
Then $|E|=\d|Q|$, and $|y|\ge a_{j_0}+\d h$ for every $y\in Q\setminus E$.
On the other hand, for every $x\in Q$,
$$|x|\le |\xi|+\sqrt{n}h\le \sqrt{n}(a_{j_0}+h).$$
Hence, if $x\in Q$ and $y\in Q\setminus E$, then
$$\frac{|x|}{|y|}\le \sqrt{n}\frac{a_{j_0}+h}{a_{j_0}+\d h}\le \frac{\sqrt n}{\d}.$$

If $a_{j_0}+h\le 0$, take
$$E:=\prod_{j\not=j_0}(a_j,a_j+h)\times (a_{j_0}+(1-\d)h,a_{j_0}+h).$$
Then $|E|=\d|Q|$ and for every $y\in Q\setminus E$,
$$|y|\ge |a_{j_0}+(1-\d)h|=|a_{j_0}+h|+\d h.$$
Therefore, for all $x\in Q$ and $y\in Q\setminus E$,
$$\frac{|x|}{|y|}\le \sqrt{n}\frac{|a_{j_0}+h|+h}{|a_{j_0}+h|+\d h}\le \frac{\sqrt n}{\d}.$$

Consider now the case where $Q$ is not contained in one of $2^n$ quadrants. Observe that since $0\not\in Q$, such a situation is possible when $n\ge 2$.
Then $Q$ has non-empty intersection with one of the coordinate axes, for example, with $x=x_{j_0}$.
There are two cases. If $Q$ is contained in the upper half space
$$\{(x_1,\dots,x_n)\in {\mathbb R}^n: x_{j_0}\ge 0\},$$
then the coordinates of the closest point $\xi\in \partial Q$ to the origin will be
$$
\xi_j=
\begin{cases}
a_{j_0},& j=j_0\\
0,& j\not=j_0,
\end{cases}
$$
where $a_{j_0}\ge 0$.
If $Q$ is contained in the lower half space
$$\{(x_1,\dots,x_n)\in {\mathbb R}^n: x_{j_0}\le 0\},$$
then
$$
\xi_j=
\begin{cases}
a_{j_0}+h,& j=j_0\\
0,& j\not=j_0,
\end{cases}
$$
where $a_{j_0}+h\le 0$.
In both cases we can take the same sets $E$ as above, and the same estimates hold.
Therefore, the proof is complete.
\end{proof}

We are now ready to prove the main result of this section.

\begin{theorem}\label{ussc}
Assume that $1<p_-\le p_+<\infty$ and that $p(\cdot)$ is a radial exponent on ${\mathbb R}^n$. Suppose that there exist $N>1$
and $\a$ satisfying
$$0<\a<\frac{p_-\min(1,p_--1)}{n}$$
such that for all $|x|\ge |y|\ge N$,
\begin{equation}\label{condd}
|p(x)-p(y)|\le \a\frac{1}{\log |y|}\log\frac{|x|}{|y|}.
\end{equation}
Then $p(\cdot)\in {\mathcal U}_{\infty}$.
\end{theorem}

\begin{proof}
Take $r>1$ and $\e>0$ such that
$$\nu:=\frac{1}{\a(1+\e)r}p_-\min(1,p_--1)>n.$$

Let ${\mathcal F}$ be a finite family of pairwise disjoint cubes. Denote
$${\mathcal F}_1:=\{Q\in {\mathcal F}:Q\subset {\mathbb R}^n\setminus [-N,N]^n\},$$
$${\mathcal F}_2:=\{Q\in {\mathcal F}:Q\subset [-N,N]^n\}\quad\text{and}\quad {\mathcal F}_3:={\mathcal F}\setminus {\mathcal F}_1\cup{\mathcal F}_2.$$

Suppose that $Q\in {\mathcal F}_1$. By Lemma \ref{cubes}, there exist subsets $E,G\subset Q$ with $|E|=\la|Q|$ and $|G|=\tau|Q|$ such that
\begin{equation}\label{xy}
\frac{|x|}{|y|}\le \frac{\sqrt{n}}{\tau}\quad(x\in Q, y\in Q\setminus G)
\end{equation}
and
\begin{equation}\label{yx}
\frac{|y|}{|x|}\le \frac{\sqrt{n}}{\la}\quad(y\in Q, x\in Q\setminus E).
\end{equation}

Take $\ga_1$ such that $\log{\sqrt{n}}=\e\log\frac{1}{\ga_1}$, and assume that $0<\la,\tau<\ga_1$.
Then, setting
$$\D_1:=\{(x,y)\in Q\times Q:|x|>|y|\},\quad \D_2:=\{(x,y)\in Q\times Q:|x|\le|y|\},$$
and using (\ref{condd}) along with (\ref{xy}) and (\ref{yx}), we obtain
$$
|p(x)-p(y)|\le
\begin{cases}
(1+\e)\a\frac{1}{\log |y|}\log(1/\tau),&(x,y)\in \big(Q\times (Q\setminus G)\big)\cap \D_1\\
(1+\e)\a\frac{1}{\log |x|}\log(1/\la),&(x,y)\in \big((Q\setminus E)\times Q\big)\cap \D_2.
\end{cases}
$$
From this,
\begin{eqnarray*}
&&\int_{Q\setminus G}\int_{Q\setminus E}F_{p,\la,\tau}(x,y)^{1/r}dxdy\\
&&\le \int_{Q\setminus G}\int_{Q\setminus E}(\tau^{p_-}\la^{p_-(p_--1)})^{\frac{1}{r|p(x)-p(y)|}}dxdy\\
&&\le \int_Q\int_{Q\setminus G}\tau^{\frac{p_-}{(1+\e)\a r}\frac{\log |y|}{\log(1/\tau)}}\chi_{\D_1}dydx+\int_{Q}\int_{Q\setminus E}\la^{\frac{p_-(p_--1)}{(1+\e)\a r}\frac{\log |x|}{\log(1/\la)}}\chi_{\D_2}dxdy\\
&&\le \int_Q\int_{Q\setminus G}\frac{1}{|y|^{\nu}}dydx+\int_{Q}\int_{Q\setminus E}\frac{1}{|x|^{\nu}}dxdy\le 2|Q|\int_{Q}\frac{1}{{|x|^{\nu}}}dx.
\end{eqnarray*}
Therefore,
\begin{eqnarray}
&&\sum_{Q\in {\mathcal F}_1}\frac{1}{|Q|}
\inf_{E\subset Q:|E|=\la|Q|\atop G\subset Q:|G|=\tau|Q|}\int_{Q\setminus G}\int_{Q\setminus E}F_{p,\la,\tau}(x,y)^{1/r}dxdy\label{partf1}\\
&&\le 2\sum_{Q\in {\mathcal F}_1}\int_Q\frac{1}{{|x|^{\nu}}}dx\le 2\int_{{\mathbb R}^n\setminus [-N,N]^n}\frac{1}{{|x|^{\nu}}}dx.\nonumber
\end{eqnarray}

Suppose now that $Q\in {\mathcal F}_2$. Then we trivially have
$$
\int_{Q}\int_{Q}F_{p,\la,\tau}(x,y)^{1/r}dxdy\le |Q|^2,
$$
and hence,
\begin{eqnarray}
&&\sum_{Q\in {\mathcal F}_2}\frac{1}{|Q|}
\inf_{E\subset Q:|E|=\la|Q|\atop G\subset Q:|G|=\tau|Q|}\int_{Q\setminus G}\int_{Q\setminus E}F_{p,\la,\tau}(x,y)^{1/r}dxdy\label{partf2}\\
&&\le \sum_{Q\in {\mathcal F}_2}|Q|\le (2N)^n.\nonumber
\end{eqnarray}

Suppose that $Q\in {\mathcal F}_3$. Consider the case where $Q:=(-a,a)^n$ for $a>N$. Then $Q$ is the only cube in ${\mathcal F}_3$.
Write $Q=P_N\cup R$, where $P_N:=[-N,N]^n$ and $R:=Q\setminus P_N$.
Then, for $E,G\subset Q$ to be chosen later,
\begin{eqnarray}
&&\int_{Q\setminus G}\int_{Q\setminus E}F_{p,\la,\tau}(x,y)^{1/r}dxdy\le\int_{R\setminus G}\int_{R\setminus E}F_{p,\la,\tau}(x,y)^{1/r}dxdy\label{np}\\
&&+\int_{P_N}\int_RF_{p,\la,\tau}(x,y)^{1/r}dxdy+\int_{R}\int_{P_N}F_{p,\la,\tau}(x,y)^{1/r}dxdy\nonumber\\
&&+\int_{P_N}\int_{P_N}F_{p,\la,\tau}(x,y)^{1/r}dxdy.\nonumber
\end{eqnarray}
Next,
$$
\int_{P_N}\int_RF_{p,\la,\tau}(x,y)^{1/r}dxdy+\int_{R}\int_{P_N}F_{p,\la,\tau}(x,y)^{1/r}dxdy\le 2(2N)^n|Q|
$$
and
$$
\int_{P_N}\int_{P_N}F_{p,\la,\tau}(x,y)^{1/r}dxdy\le |P_N|^2\le (2N)^n|Q|.
$$
Hence,
\begin{eqnarray*}
&&\int_{Q\setminus G}\int_{Q\setminus E}F_{p,\la,\tau}(x,y)^{1/r}dxdy\\
&&\le\int_{R\setminus G}\int_{R\setminus E}F_{p,\la,\tau}(x,y)^{1/r}dxdy+3(2N)^n|Q|.
\end{eqnarray*}

Take now $G:=(-\tau^{1/n}a,\tau^{1/n}a)^n$ and $E:=(-\la^{1/n}a,\la^{1/n}a)^n$. Then $|G|=\tau|Q|$ and $|E|=\la|Q|$.
For $0<\d<1$ denote
$$S_{a,\d}:=(0,a)^n\setminus \big(0,\max(N,\d^{1/n}a)\big)^n.$$
Then, using that $p(\cdot)$ is radial, we obtain
$$\int_{R\setminus G}\int_{R\setminus E}F_{p,\la,\tau}(x,y)^{1/r}dxdy=4^n\int_{S_{a,\tau}}\int_{S_{a,\la}}F_{p,\la,\tau}(x,y)^{1/r}dxdy.$$

For all $x\in (0,a)^n$ and $y\in S_{a,\tau}$,
$$\frac{|x|}{|y|}\le \frac{\sqrt{n}a}{\tau^{1/n}a}=\frac{\sqrt{n}}{\tau^{1/n}}.$$
Take $\ga_2$ such that for all $0<\tau<\ga_2$.
$$\log\frac{\sqrt{n}}{\tau^{1/n}}\le (1+\e)\log(1/\tau).$$
Then, by (\ref{condd}), for all $x\in (0,a)^n$ and $y\in S_{a,\tau}$ such that $|x|\ge |y|$ we have
$$
|p(x)-p(y)|\le (1+\e)\a\frac{1}{\log|y|}\log(1/\tau).
$$
Similarly, for $0<\la<\ga_2$ and for all $y\in (0,a)^n$ and $x\in S_{a,\la}$ such that $|x|\le |y|$ we have
$$
|p(x)-p(y)|\le (1+\e)\a\frac{1}{\log|x|}\log(1/\la).
$$
Therefore, exactly as in the case for $Q\in {\mathcal F}_1$, we obtain
\begin{eqnarray*}
\int_{S_{a,\tau}}\int_{S_{a,\la}}F_{p,\la,\tau}(x,y)^{1/r}dxdy&\le& 2a^n\int_{(0,a)^n\setminus (0,N)^n}\frac{1}{|x|^{\nu}}dx\\
&\le& |Q|\int_{(0,\infty)^n\setminus (0,N)^n}\frac{1}{|x|^{\nu}}dx.
\end{eqnarray*}

Combining this with the previous estimates yields
$$\int_{Q\setminus G}\int_{Q\setminus E}F_{p,\la,\tau}(x,y)^{1/r}dxdy\le 3(2N)^n|Q|+4^n|Q|\int_{(0,\infty)^n\setminus (0,N)^n}\frac{1}{|x|^{\nu}}dx,$$
and hence,
$$\frac{1}{|Q|}\int_{Q\setminus G}\int_{Q\setminus E}F_{p,\la,\tau}(x,y)^{1/r}dxdy\le C_{n,N}.$$

In the general case of $Q\in {\mathcal F}_3$, split ${\mathcal F}_3:={\mathcal G}_1\cup {\mathcal G}_2$,
where
$${\mathcal G}_1:=\{Q\in {\mathcal F}_3:\ell_Q\le N\}$$
and
$${\mathcal G}_2:=\{Q\in {\mathcal F}_3:\ell_Q>N\},$$
where $\ell_Q$ denotes the side length of $Q$. For $Q\in {\mathcal G}_1$ we have that $Q\subset [-2N,2N]^n$, and hence the same argument as for ${\mathcal F}_2$ applies.
On the other hand, there are at most $k_n$ cubes in ${\mathcal G}_2$ and for every $Q\in {\mathcal G}_2$ there exists $a>N$ such that $Q\subset (-a,a)^n$ and $a\le c_n\ell_Q$.
Hence, for each cube $Q\in {\mathcal G}_2$ we can use the same argument as just considered in the case where $Q=(-a,a)^n$ for $a>N$. Therefore,
$$\sum_{Q\in {\mathcal F}_3}\frac{1}{|Q|}
\inf_{E\subset Q:|E|=\la|Q|\atop G\subset Q:|G|=\tau|Q|}\int_{Q\setminus G}\int_{Q\setminus E}F_{p,\la,\tau}(x,y)^{1/r}dxdy\le C_{n,N}.$$

From this, assuming that $0<\la,\tau<\min(\ga_1,\ga_2)$, and using (\ref{partf1}) and (\ref{partf2}), we obtain the desired uniform estimate for $\sum_{Q\in {\mathcal F}}$,
which completes the proof.
\end{proof}

\begin{cor}\label{usesc} Let $s(\cdot)$ be a function on $[0,\infty)$ with $1<s_-\le s_+<\infty$, and let $p(x):=s(|x|)$.
Suppose that $p(\cdot)\in A_{p(\cdot)}$, and there exist $N>1$ and $\a$ satisfying
\begin{equation}\label{alsn}
0<\a<\frac{s_-\min(1,s_--1)}{n}
\end{equation}
such that for all $t\ge N$,
\begin{equation}\label{der}
\Big|\frac{d}{dt}s(t)\Big|\le \frac{\a}{t\log t}.
\end{equation}
Then $p(\cdot)\in {\mathcal P}$.
\end{cor}

\begin{proof}
It follows from (\ref{der}) that for $|x|>|y|\ge N$,
$$
|p(x)-p(y)|\le \a\int_{|y|}^{|x|}\frac{1}{t\log t}dt\le \a\frac{1}{\log |y|}\log\frac{|x|}{|y|}.
$$
Hence, by Theorem \ref{ussc}, $p(\cdot)\in {\mathcal U}_{\infty}$, which, by Theorem \ref{mr}, implies $p(\cdot)\in {\mathcal P}$.
\end{proof}

\begin{remark}\label{ontan}
By Proposition \ref{stlh}, the $A_{p(\cdot)}$ condition in Corollary \ref{usesc} can be replaced by the $LH_0$ condition.
\end{remark}

\begin{example}\label{exx1}
Let $t_0\ge {\rm{e}}$ and let $\f$ be an arbitrary non-decreasing, positive and unbounded function on $[t_0,\infty)$. Define
$$
s(t):=
\begin{cases}
c+\int_{t_0}^t\frac{\sin y}{y\log y\f(y)}dy,&t>t_0\\
c,&t\in [0,t_0],
\end{cases}
$$
where $c>0$ is chosen in such a way to have $s_->1$. Then $p(x):=s(|x|)\in {\mathcal P}$.

Indeed, in this case, for every $\a>0$ there exists $N>1$ such that (\ref{der}) is satisfied for all $t\ge N$. Also, we obviously have that $p(\cdot)\in LH_0$.
Therefore, by Corollary \ref{usesc} along with Remark \ref{ontan}, $p(\cdot)\in {\mathcal P}$.

One can also avoid the function $\f$, setting
$$
s(t):=
\begin{cases}
c+\a\int_{t_0}^t\frac{\sin y}{y\log y}dy,&t>t_0\\
c,&t\in [0,t_0],
\end{cases}
$$
where $c$ and $\a$ are chosen in such a way that condition (\ref{alsn}) holds. Then $p(x):=s(|x|)\in {\mathcal P}$.
\end{example}

\begin{example}\label{exx2} For $c,\a>0$ define
$$s(t):=
\begin{cases}
c+\a\sin(\log\log t),&t>{\rm{e}}\\
c,&t\in [0,{\rm{e}}].
\end{cases}
$$

Then $s_-=c-\a$. Therefore, if
$$
0<\a<\frac{(c-\a)\min(1,c-\a-1)}{n},
$$
then conditions (\ref{alsn}) and (\ref{der}) hold, and hence,
$$p(x):=c+\a\sin(\log\log |x|)\chi_{\{|x|>{\rm{e}}\}}(x)\in {\mathcal P}.$$
For instance, one can take here $c=\a+2$ and $0<\a<2/n$.

Observe that this example but with less precise values of $c$ and $\a$ was obtained in \cite{L1} using a
completely different approach based on the pointwise multipliers for $BMO$.
\end{example}

There is a more general version of Theorem \ref{ussc}, allowing to change $p(\cdot)$ on a set of finite measure.

\begin{theorem}\label{ussc1}
Let $N>1$ and let $A\subset {\mathbb R}^n\setminus [-N,N]^n$ be a symmetric (in the sense that the function $\chi_A$ is radial) open set of finite measure.
Assume that $1<p_-\le p_+<\infty$ and that $p(\cdot)$ is a radial exponent on~${\mathbb R}^n$.
Suppose that there exists an $\a$ satisfying
$$0<\a<\frac{p_-\min(1,p_--1)}{n}$$
such that for all $x,y\in {\mathbb R}^n\setminus ([-N,N]^n\cup A)$ with $|x|\ge |y|$,
\begin{equation}\label{condd1}
|p(x)-p(y)|\le \a\frac{1}{\log |y|}\log\frac{|x|}{|y|}.
\end{equation}
Then $p(\cdot)\in {\mathcal U}_{\infty}$.
\end{theorem}

\begin{remark}\label{expls}
The difference with Theorem \ref{ussc} is that condition (\ref{condd1}) holds now on a smaller set, that is, we allow to define $p(\cdot)$ arbitrary on a set of finite measure, although we still require that $p(\cdot)$ is radial. In the case where the set $A$ is bounded, Theorem \ref{ussc1} follows trivially from Theorem \ref{ussc}. Therefore, the interesting case is where $A$ is unbounded.
\end{remark}

\begin{proof}[Proof of Theorem \ref{ussc1}]
The proof is very similar to the previous one, and therefore we only indicate the necessary changes.

Let ${\mathcal F}$ be a finite family of pairwise disjoint cubes. Denote
$${\mathcal F}_1:=\{Q\in {\mathcal F}:Q\subset {\mathbb R}^n\setminus ([-N,N]^n\cup A)\},$$
$${\mathcal F}_2:=\{Q\in {\mathcal F}:Q\subset [-N,N]^n\cup A\}\quad\text{and}\quad {\mathcal F}_3={\mathcal F}\setminus {\mathcal F}_1\cup{\mathcal F}_2.$$

For ${\mathcal F}_1$ and ${\mathcal F}_2$ the argument is exactly the same as in the previous proof. In particular, it is important for ${\mathcal F}_2$ that the set $A$ is of finite measure.

Consider ${\mathcal F}_3$. Assume that $Q\in {\mathcal F}_3$ is such that $Q\subset {\mathbb R}^n\setminus [-N,N]^n$.
In this case we remove from $Q$ the same sets $E$ and $G$ as in the previous proof for the cubes from ${\mathcal F}_1$. Next, write $Q=P\cup R$, where
$P:=Q\cap A$ and $R:=Q\setminus A$. Then, using the same estimate as in (\ref{np}), we obtain
$$\int_{Q\setminus G}\int_{Q\setminus E}F_{p,\la,\tau}(x,y)^{1/r}dxdy\le \int_{R\setminus G}\int_{R\setminus E}F_{p,\la,\tau}(x,y)^{1/r}dxdy+3|A||Q|.$$
The integral on the right-hand side is bounded exactly in the same way as in the previous proof for $Q\in {\mathcal F}_1$ instead of $R$. We obtain
$$\int_{R\setminus G}\int_{R\setminus E}F_{p,\la,\tau}(x,y)^{1/r}dxdy\le 2|Q|\int_{Q}\frac{1}{{|x|^{\nu}}}dx.$$
Combining this with the previous estimate yields
\begin{eqnarray*}
&&\sum_{Q\in {\mathcal F}_3:Q\subset {\mathbb R}^n\setminus [-N,N]^n}\frac{1}{|Q|}
\inf_{E\subset Q:|E|=\la|Q|\atop G\subset Q:|G|=\tau|Q|}\int_{Q\setminus G}\int_{Q\setminus E}F_{p,\la,\tau}(x,y)^{1/r}dxdy\\
&&\le 2\int_{{\mathbb R}^n\setminus [-N,N]^n}\frac{1}{{|x|^{\nu}}}dx+3|A|.
\end{eqnarray*}

Suppose now that $Q\in {\mathcal F}_3$ is such that $Q\cap [-N,N]^n\not=\emptyset$. Then the same reduction as in the previous proof shows that it is enough to consider the case where $Q:=(-a,a)^n$ for $a>N$.
Applying the same estimate as in (\ref{np}) with $P:=Q\cap (A\cup [-N,N]^n)$ and $R:=Q\setminus P$, we obtain
\begin{eqnarray*}
&&\int_{Q\setminus G}\int_{Q\setminus E}F_{p,\la,\tau}(x,y)^{1/r}dxdy\\
&&\le\int_{R\setminus G}\int_{R\setminus E}F_{p,\la,\tau}(x,y)^{1/r}dxdy+3\big((2N)^n+|A|\big)|Q|.
\end{eqnarray*}
Next, take the same $G$ and $E$ as in the previous proof. Then, setting
$$S_{a,\d}:=(0,a)^n\setminus \Big(\big(0,\max(N,\d^{1/n}a)\big)^n\cup A\Big)$$
yields
$$\int_{R\setminus G}\int_{R\setminus E}F_{p,\la,\tau}(x,y)^{1/r}dxdy=4^n\int_{S_{a,\tau}}\int_{S_{a,\la}}F_{p,\la,\tau}(x,y)^{1/r}dxdy.$$
The rest is exactly the same as in the previous proof. We obtain
\begin{eqnarray*}
&&\int_{Q\setminus G}\int_{Q\setminus E}F_{p,\la,\tau}(x,y)^{1/r}dxdy\\
&&\le 3\big((2N)^n+|A|\big)|Q|+4^n|Q|\int_{(0,\infty)^n\setminus (0,N)^n}\frac{1}{|x|^{\nu}}dx,
\end{eqnarray*}
which, along with the previous cases, completes the proof.
\end{proof}

\subsection{A comparison with Nekvinda's condition}\label{ss63}
Let us introduce the following notation. Denote
$${\rm e}_0:=1\quad\text{and}\quad {\rm e}_{k+1}:={\rm e}^{{\rm e}_k}\quad (k\in {\mathbb N}),$$
and define functions $\log_kx$ on $({\rm e}_k,\infty)$ by
$$\log_0x:=x\quad\text{and}\quad \log_{k+1}x:=\log(\log_kx).$$
Denote also, for $\a>0$,
$$b_{k,\a}(x):=-\frac{1}{\a}\frac{d}{dx}(\log_k^{-\a}x).$$

In \cite{Ne4}, Nekvinda obtained the following result.

\begin{theorem}\label{nekth} Assume that $1<p_-\le p_+<\infty$, and $p(\cdot)\in LH_0$. Next, let $s(\cdot)$ be a function of $[0,\infty)$ with $1<s_-\le s_+<\infty$ and let $q(x):=s(|x|)$.
Assume that there exist $C>0$, $k\in {\mathbb N}$ and $0<\a<1$ such that
\begin{equation}\label{ncond1}
\Big|\frac{d}{dt}s(t)\Big|\le Cb_{k,\a}(t)\quad(t\ge {\rm e}_k).
\end{equation}
Further, assume that, for some $0<c<1$,
\begin{equation}\label{ncond2}
\int_{{\mathbb R}^n}c^{\frac{1}{|p(x)-q(x)|}}dx<\infty.
\end{equation}
Then $p(\cdot)\in {\mathcal P}$.
\end{theorem}

This result is an extension of \cite{Ne3}, where the same was established assuming that $s(\cdot)$ is monotone.

The key example of interest here is $p(x)=q(x)$. Then (\ref{ncond2}) holds trivially, and Theorem \ref{nekth} says that a radial exponent $p(x)=s(|x|)$ belongs to ${\mathcal P}$ if $p(\cdot)\in LH_0$
and the behaviour of $p(\cdot)$ at infinity is controlled by a rather weak assumption expressed in (\ref{ncond1}). In general, (\ref{ncond2}) says that the behaviour of $p(\cdot)$ at infinity may be slightly perturbed, for example, it is enough to assume that
$$|p(x)-q(x)|\le \frac{C}{\log|x|}$$
for $|x|\ge N$ for some large $N$, where $q(\cdot)$ is radial and satisfies (\ref{ncond1}).

We consider the case where $p(x)=q(x)$.
Observe that for any $k\in {\mathbb N}$, the function $b_{k,\a}(t)$ can be written in the form
$$b_{k,\a}(t)=\frac{1}{t\log t\f_{k,\a}(t)},$$
where
$$
\f_{k,\a}(t):=
\begin{cases}
\log^{\a}t,&k=1\\
\log\log^{1+\a}t,&k=2\\
\Big(\prod_{j=2}^{k-1}\log_jt\Big)\log_k^{1+\a}t,&k\ge 3.
\end{cases}
$$

Therefore, the assumption (\ref{ncond1}) is stronger than assumption (\ref{der}) in Corollary \ref{usesc}.
In particular, for an arbitrary increasing function $\f$ on $(0,\infty)$ such that
$$\lim_{t\to \infty}\frac{\f(t)}{\log\log t}=0,$$
the exponent $p(\cdot)\in {\mathcal P}$ from Example \ref{exx1} cannot be handled, in general, by means of Theorem \ref{nekth}.

\end{document}